\documentclass[12pt]{article}
\usepackage{epsfig,psfrag,amsmath,amssymb,latexsym}
\usepackage{amscd }
\usepackage{amsfonts}
\usepackage{graphicx}
\renewcommand{\arraystretch}{1}

\newtheorem{theorem}{Theorem}[section]

\newtheorem{lemma}{Lemma}[section]

\newenvironment{proof}[1][Proof]{\par\noindent\textbf{#1.} }{\hfill~\rule{0.5em}{0.5em}}

\newcommand{\E}{\mathbb{E}}
\renewcommand{\P}{\mathbb{P}}
\def\l{\ell}
\numberwithin{equation}{section}

\def\var{\mathop{\rm Var}}
\begin{document}
\title{Sparse Long Blocks and the Micro-Structure of the Longest Common Subsequences}
\author{S.~Amsalu\thanks{Department of Information Science, 
Faculty of Informatics,
University of Addis Ababa, Ethiopia}
 \and C.~Houdr\'e
\thanks{\ \ \ \ School of Mathematics, Georgia
 Institute of Technology, Atlanta, GA 30332, USA, houdre@math.gatech.edu. 
Research supported in part by the 
grant \#246283 from the Simons Foundation.  Many thanks to the LPMA of the Universit\'e Pierre et Marie 
Curie, where part of this research was carried out, for its hospitality while supported by a Simons Foundation Fellowship grant \#267336. } 
\and H.~Matzinger\thanks{\ \ \ \ School of Mathematics, Georgia
 Institute of Technology, Atlanta, GA 30332, USA, matzi@math.gatech.edu. Research Supported in part by
NSA Grant H98230-09-1-0017.} }
\maketitle

\begin{abstract} 
Consider two random strings having the same length and
generated by an iid sequence taking its values uniformly 
in a fixed finite alphabet.  Artificially place a long constant block 
into one of the strings, where a constant block is a contiguous 
substring consisting only 
of one type of symbol. The long block replaces a segment 
of equal size and  its length 
is smaller than the length of the strings, but larger than its square-root.  
We show that for sufficiently long strings 
the optimal alignment corresponding to a Longest Common Subsequence 
(LCS) treats the inserted block very differently 
depending on the size of the alphabet. For two-letter alphabets, 
the long constant block gets mainly aligned with the same symbol
from the other string, while for three or more letters the opposite is
true and the block gets mainly aligned with gaps.

We further provide simulation results on the proportion of gaps
in blocks of various lengths.  In our simulations, the blocks are
``regular blocks" in an iid sequence, and are not artificially 
inserted.  Nonetheless, we observe for these natural blocks 
a phenomenon similar to the one shown in case of 
artificially-inserted blocks: with two letters,
the long blocks get aligned with a smaller proportion of gaps; 
for three or more letters, the opposite is true.  

It thus appears that the 
microscopic nature of two-letter optimal alignments and three-letter
optimal alignments are entirely different from each other.
\end{abstract}

\section{Introduction}
Let $x$ and $y$ be two finite strings.  
A common subsequence of $x$ and $y$ is a subsequence which is 
a subsequence of $x$ and at the same time a subsequence of $y$, while a 
Longest Common Subsequence (LCS) of $x$ and $y$ 
is a common subsequence of maximal length.  

A LCS is often used as a measure of strings relatedness, 
and can be viewed as an alignment  
aligning same letter pairs.  Every such 
alignment defines a common subsequence, and 
the length of the subsequence corresponding 
to an alignment, i.e., the number of aligned letter-pairs, is called 
the {\it score of the alignment}.  
The alignment representing a LCS is said to be {\it optimal} or called 
an {\it optimal alignment}.

Longest Common Subsequences (LCS) and Optimal Alignments (OA) 
are important tools used in Computational Biology and Computational Linguistics 
for string matching \cite{Capocelli1}, \cite{watermanintrocompbio}, 
\cite{Watermangeneralintro}, and, in particular, 
for the automatic recognition of related DNA pieces. 

The asymptotic behavior of the expectation and of the variance of the length 
of the LCSs of two independent random strings 
has been studied, among others, by probabilists, physicists, 
computer scientists and computational biologists. 
The LCS problem can be formulated 
as a last passage percolation problem with 
dependent weights; and finding the order of the fluctuations in such 
percolation problems has been open for quite a while.

Throughout, 
$LC_n:=|LCS(X_1X_2\ldots X_n;Y_1Y_2\ldots Y_n)|$ is the length of the LCSs of 
two random strings where $\{X_n\}_{n\ge 1}$ and $\{Y_n\}_{n\ge 1}$ are 
two independent iid sequences uniformly distributed on 
an fixed alphabet of size $k$.   Clearly, $LC_n$ is super-additive 
and via the sub-additive ergodic theorem, 
Chv\'atal and Sankoff \cite{Sankoff1} showed (for stationary sequences) 
that 
$$\gamma^*_k:=\lim_{n\rightarrow\infty}\frac{\E\, LC_n}{n}.$$  
However, even for the simplest distributions such as for binary 
equiprobable alphabet, the exact value 
of $\gamma^*_k$  is unknown.  Nevertheless, extensive simulations have led to 
very good approximate values for these constants, e.g., in the iid case, 
\begin{equation}
\label{table}\begin{array}{cc|c|c|c|c}
k       & &2   &3&4&\cdots \\\hline
\gamma^*_k& &0.812 &0.717 &0.654 &\cdots 
\end{array}
\end{equation}
where the precision in the above table is around $\pm 0.01$ (see \cite{boutet}).  
Exact lower and upper bounds have also been obtained,  
an overview of those as well as new bounds are available in \cite{lueker}.   

Alexander~\cite{Alexander1} further established the 
speed of convergence of $\E\, LC_n/n$ to $\gamma^*_k$, for iid sequences, showing that  
\begin{equation}\label{alex}
\gamma^*_k n-C_L\sqrt{n\log n}\le \E\, LC_n\le\gamma^*_k n,
\end{equation}
where $C_L>0$ is a constant depending neither on 
$n$ nor on the distribution of the strings.  
 
Below, we also need to consider two sequences of different lengths 
but such that the two lengths are in a fixed proportion to each 
other.  To do so, for $p\in(-1,1)$, let  
\begin{equation}
\label{gammaknp}\gamma_k(n,p):=
\frac{\E|LCS(X_1X_2\ldots X_{n-np};Y_1Y_2\ldots Y_{n+np})|}{n},
\end{equation}
where above, when real, the indices are understood to be roundings to the nearest 
positive integers, and let 
\begin{equation}
\label{meancurve}
\gamma_k(p):=
\lim_{n\rightarrow\infty}\gamma_k(n,p),
\end{equation}
which is again finite by standard super-additivity arguments.  
The function $\gamma_k: p\mapsto \gamma_k(p)$ is called the {\it mean LCS-function}; 
it is clearly bounded, non-negative, symmetric around $p=0$, and, as shown next, 
concave; therefore it has a maximum at $p=0$.   To prove the concavity 
property of $\gamma_k$, first by super-additivity,   
\begin{align*} n\gamma_k\left(n, \frac{p+q}{2}\right) & = 
\E|LCS(X_1\ldots X_{n(1-(p+q)/2)};Y_1\ldots Y_{n(1+(p+q)/2)})| \\ 
&\ge  \E|LCS(X_1\ldots X_{n(1-p)/2};Y_1\ldots Y_{n(1+p)/2})| \\ 
&\quad \quad +   \E|LCS(X_1\ldots X_{n(1-q)/2};Y_1\ldots Y_{n(1+q)/2})|\\ 
&= \frac{n}{2}\gamma_k\left(\frac{n}{2},p\right)+ \frac{n}{2}\gamma_k\left(\frac{n}{2},q\right).  
\end{align*}
Therefore,   
$$\gamma_k\left(\frac{p+q}{2}\right)\ge \frac{1}{2}\gamma_k\left(\frac{n}{2},p\right)+
\frac{1}{2}\gamma_k\left(\frac{n}{2},q\right),$$  
which by taking limits, as $n\rightarrow \infty$, leads  
$$\gamma_k\left(\frac{p+q}{2}\right)\geq \frac{\gamma_k(p)+\gamma_k(q)}{2}.$$

The function $\gamma_k$ corresponds to the 
wet-region-shape in first passage percolation. From  
our simulations it seems quite clear that $\gamma_k$ is strictly 
concave in a neighborhood of $p=0$, but this might be highly non trivial to prove.  
As a matter of fact, in first passage  percolation, the 
corresponding problem, of showing the strict convexity of the asymptotic wet-region 
shape remains open.  

The main results of the present paper (Theorem~\ref{maintheoremsubsection} and 
Theorem~\ref{maintheoremsubsection2}) are concerned with sequences 
of length $n=2d$.  They describe 
the effect of replacing an iid piece, of length $d^\beta$, $1/2<\beta<1$, 
with a long constant block of equal length.  It is shown that typically 
replacing an iid part by a long constant block 
leads to a decrease in the LCS.  It is also shown that in the binary case, 
the long constant block gets mainly aligned with letters while with three 
or more letters the opposite 
is true.  To illustrate our results, 
consider the sequences
$$01{\bf000 00}001$$
and 
$$00101 11010,$$
where the bold faced letters are those of the replacing block. 
Theorem~\ref{maintheoremsubsection} and Theorem~\ref{maintheoremsubsection2} 
respectively assert that the optimal alignments behave very differently 
depending upon the size of the alphabet:  
In the binary case the long constant block gets mainly aligned 
with bits, while with three or more equiprobable letters it gets mainly aligned 
with gaps.  This phenomenon holds with high probability and 
assuming $d$ to be sufficiently large.  
In the above example, a (non unique) LCS is given by $00000$ and it  
corresponds to the (non-unique) optimal alignment 
$$
\begin{array}{c|c|c|c|c| c|c|c|c|c| c|c|c|c|c }
0&1&{\bf0}&{\bf0}& &{\bf0}& &  & &{\bf 0}& &{\bf0}&0&0&1\\\hline
 & &0     &0     &1&0     &1& 1&1&0      &1&0     & & &
\end{array}
$$
In the above example of optimal alignment all 
the zeros from the long constant block (in bold face) got aligned with 
zeros and not with gaps.  Our results show that for binary 
sequences, the artificially inserted 
long constant block gets aligned with very few gaps; more precisely, 
the number of gaps has an order of magnitude smaller than 
the length of the long block.  With three or more letters the opposite is true and 
the long constant block gets aligned almost exclusively
with gaps.  The situation for three or more letters is 
not surprising, but the binary one is rather counter-intuitive.  
Although our proof is for $d$ going to infinity, 
this phenomenon is observed in simulations for regular 
blocks which have not been artificially inserted: 
with binary sequences longer blocks tend to be 
aligned with a small proportion of gaps, while with more letter the opposite is true.
(More examples of this type are given at the beginning of 
Section~\ref{statements}.)   We thus seem to have uncovered an interesting 
phenomenon, in that the microstructure of the optimal alignment of iid 
sequences for binary sequences is fundamentally different 
from the case with more letters. 
It is another instance (see \cite{HoLi}) where the size of the alphabet in 
a subsequence problem plays an important role.  

Finally, let us described some differentiability conditions  
on $\gamma_k$ which could be used to  
obtain our results.  First, since it is concave, $\gamma_k$ 
has non-increasing left and a right derivatives at any $p\in (-1,1)$, 
with $\gamma_k^\prime(p^-) \ge \gamma_k^\prime(p^+)$, while by symmetry, 
$\gamma_k^\prime(p^\pm)=-\gamma_k^\prime((-p)^\mp)$.   

Next, let $0\le p_M < 1$ be 
the largest real for which  
$\gamma_k$ is maximal. Hence, $[-p_M,p_M]$  
is the largest interval on which $\gamma_k$ is everywhere 
equal to its maximal value $\gamma_k(0)$, i.e.,  
$[-p_M,p_M]=\gamma^{-1}(\{\gamma_k(0)\})$.

Our theorems will be verified under any one of the following four conditions:

\begin{enumerate}
\item{} The mean LCS-function $\gamma_k$ is strictly concave in a neighborhood of the origin and is 
differentiable at $0$ (and so $p_M=0$ and $\gamma_k^\prime(0) = 0$).  
\item{}The function $\gamma_k$ is differentiable at $p_M$, i.e., $\gamma_k^\prime(p_M^+) = 
\gamma_k^\prime(p_M^-)$ 
and therefore (either by symmetry or since $\gamma_k^\prime(p_M^-) = 0$ 
if $p_M>0$ ) $\gamma_k^\prime(p_M) = 0$.  
\item{} The absolute value of $\gamma_k^\prime(p_M^+)\le 0$ is dominated 
by the absolute value of $\gamma_k(0)-(2/k)$:    
$$\left|\frac{\gamma_k^\prime(p_M^+)}{2}\right|<\left|\frac{\gamma_k(0)}{2}-\frac{1}{k}\right|.
$$
\item{} The function $\gamma_k$ is strictly concave in a neighborhood of the origin 
and its right derivative at the origin is such that:   
$$\left|\frac{\gamma_k^\prime(0^+)}{2}\right|<\left|\frac{\gamma_k(0)}{2}-\frac{1}{k}\right|.
$$

\end{enumerate}

Clearly, $1 \implies 2\implies 3, 1\implies 4$.  
In the present article, the main results are proved under 
the assumptions of Condition 2.  But, in the summary of the proofs (Section 2) it is 
indicated how Condition 3 or 4 would also work.  
With Condition 3, the notations for the proofs would become even more cumbersome 
since an additional term would appear everywhere.    
From our simulations, we have no doubt that $p_M=0$ and that even Condition 1 holds true. 
Note also that Condition 3, unlike the others, can be verified up 
to a certain confidence level by Monte Carlo  
simulations, making it rather nice and important.  

As for the content of the paper,  
Section~\ref{mainideas} presents 
some of the main ideas behind the proofs, while 
statements of the main results are given in Section~\ref{statements}.   
Section~\ref{simulations} presents many simulations and discusses the nature of two-letter 
and three-letter optimal alignments.  
The proofs of the main results are presented 
in Section~\ref{proofs} (Subsection~\ref{rigorous3} for three or more letters 
and Subsection~\ref{rigorous2} for binary alphabets.)  
In addition to its own interest, the present paper serves as background 
to showing that the variance of the LCS of two iid random strings with many 
added long blocks is linear in the length of the strings (see \cite{amhoma}).

\section{Main Ideas}\label{mainideas}
This section outlines the main ideas behind the proofs of the results.  
Below, both strings $X$ and $Y$ have length $2d$ and approximately 
in its middle, the sequence $X$ contains a long constant block of approximate length 
$\l=d^\beta$ (Actually, in many of the proofs we exactly take $\l=d^\beta$, $\l$ even, 
but it is clear that choosing a multiple of $d^\beta$ or even multiplying 
$d^\beta$ by a logarithmic factor of $d$ would work).  Since we also believe that the 
phenomena we uncovered are cogent for naturally occurring long blocks we often 
interchange the symbols $\l$ and $d^\beta$.   
The two sequences are independent and except for the long constant 
block in $X$, iid uniform.   Besides combinatorial and concentration inequalities, 
the proofs results follow from the following two facts (the first of which also follows from 
Hoeffding's martingale inequality):
\begin{enumerate}
\item{} First, 
$\gamma_k(n,p)$ converges, uniformly in $p$ to $\gamma_k(p)$, at a rate of $\sqrt{{\ln n}/{n}}$.  
More precisely, Alexander (see Example 1.4 and Theorem~4.2 in \cite{Alexander2}), 
shows that there exists a constant $C_\gamma>0$ independent of 
$n$ and $p\in(-1,1)$ such that
\begin{equation}
\label{alexander}
|\gamma_k(n,p)-\gamma_k(p)|\leq C_\gamma\sqrt{\frac{\ln n}{n}},
\end{equation}
for all $n$ and all $p\in(-1,1)$. 

\item{} Second, when a string with only one symbol gets aligned with  
another iid string with equiprobable letters, a LCS is typically much  
shorter than for two iid strings with equiprobable letters.\end{enumerate}

Let us illustrate this second point on an example. 
Let $v=000000$, $w=100101$, so $LCS(v,w)=000$ and $|LCS(v,w)|=3$ which is 
the number of zeros in the string $w$.  Now, if $w$ is an iid string with $k$ 
equiprobable letters and if $v$ consists only of zeros, both strings having the same length, 
then typically the LCS has length approximately equal to $|w|/k$.  
This is typically much less than for two iid sequences with equiprobable letters, where 
the LCS length is approximated by $\gamma^*_k |w|$.   
One then compares $1/k$ and $\gamma^*_k$ and see, that $1/k$ is smaller 
than $\gamma^*_k$ for $k\geq 2$.  

In the present article, we prove two fundamental properties of the 
optimal alignment with an inserted long constant block:
\begin{enumerate}
\item{}  
First, replacing an iid part in one of the sequences  
by a long constant block causes an expected loss of the LCS.  In fact, the expected effect 
of replacing an iid piece with a long constant block of equal length is linear in the length 
of block (as shown in Section~\ref{proofs}).  The variance cannot make up for this loss since, 
by Hoeffding's inequality, the standard deviation is at most of order $\sqrt{d}$ 
but $d^\beta$, $1/2 < \beta < 1$, has an order of magnitude greater that $\sqrt{d}$.    
\item{}   Second, and still in the above setting, depending on whether 
${\gamma^*_k}/{2}-{1}/{k}$ is positive or not, the long block gets mainly aligned with gaps or not.  
\end{enumerate}

Parts I and II which follow, outline the proof 
estimating the number of gaps aligned with the long block.  

To start with, the string $X$ is made up of 
three concatenated strings $X^a$, $B$ and $X^c$, 
where $X^a$ and $X^c$ are iid strings and $B$ is
a constant long block.  This is written as:  
$$X=X^a B X^c,$$
where $X^a$ and $X^c$ have common length equal to: $d-d^\beta/2=d+o(d)$.  
Next, let $\pi$ be an optimal alignment of $X$ with the iid string 
$Y=Y_1Y_2\ldots Y_{2d}$, and let $Y^a$, $Y^b$, 
and $Y^c$ denote pieces of $Y$ respectively aligned with 
$X^a$, $B$, and $X^c$.

Next, modify the alignment $\pi$ to obtain a new alignment 
$\bar{\pi}$.  For this, align $X^a$ with $Y^aY^b$ instead of only with 
$Y^a$. The block $B$ gets aligned exclusively with gaps under $\bar{\pi}$, 
while the alignment of $X^c$ and $Y^c$ remains unchanged.  
Request also that $\bar{\pi}$ aligns $X^a$ and $Y^aY^b$ 
in an optimal way, so that the part of the alignment 
score of $\bar{\pi}$ coming from aligning $X^a$ with $Y^aY^b$ 
is equal to $|LCS(X^a,Y^aY^b)|$.  The alignments $\pi$ and $\bar\pi$ 
are schematically represented via:  
$$\pi:\;\;\;\;\begin{array}{c|c|c}
X^a &B &X^c\\\hline
Y^a &Y^b & Y^c
\end{array}
$$ 
and
$$\bar{\pi}:\;\;\;\;\begin{array}{c|c|c}
X^a &B &X^c\\\hline
Y^aY^b & & Y^c
\end{array}
$$ 
As for the scores, they are given by: 
\begin{align*} 
&{\tt score \; of\;}\pi=|LCS(X^a,Y^a)|+|LCS(B,Y^b)|+|LCS(X^c,Y^c)|,\\
&{\tt score \; of\;}\bar{\pi}=|LCS(X^a,Y^aY^c)|+|LCS(X^c,Y^c)|.  
\end{align*}
The difference between the two alignment scores has two sources: 
first the loss of those letters of the block $B$ which where aligned 
with letters, and not with gaps, 
under $\pi$ (while under $\bar{\pi}$, all the letters get aligned with gaps).  
If $h$ denotes the length of  $Y^b$, then this expected loss is typically 
$h/k$. ($B$ is only made up of letters of one type, 
while in the iid part, each letter has probability $1/k$, 
and therefore a given letter is expected to appear $h/k$ times 
in the substring $Y^b$).  The second source of change in score between 
$\pi$ and $\bar{\pi}$ comes from ``adding $Y^c$ to the alignment 
of $X^a$ and $Y^a$''.  The amount gained is then 
\begin{equation}
\label{thisdiff}
|LCS(X^a,Y^aY^b)|-|LCS(X^a,Y^a)|.
\end{equation}
Assuming, say, that Condition 2 holds and from the optimality of $\pi$, 
it is easy to see that $Y^a$ and $Y^c$ have length 
$d+o(d)$, for $d$ large.   Assume next that 
$Y^b$ has length $h=cd^\beta$, where $c>0$ is a constant not depending 
on $d$. The increase given in \eqref{thisdiff}  
can be described as the increase in LCS, when adding $h=cd^\beta$ iid letters to 
two iid strings of length $d+o(d)$.  Part I below analyzes the size of this increase and is then used 
in Part II to explain how to estimate 
the proportion of gaps the long block gets aligned with.

\paragraph{I) Effect of adding $h=cd^\beta$ symbols to one sequence only.  }
Let $V$ and $W$ be two independent iid strings of length 
$d+o(d)$ with $k$ equiprobable letters.  Let $d_1$ and $d_2$ be the respective 
length of $V$ and $W$, and let 
\begin{align*}
\bar{d}&=\frac{d_1+d_2}{2}\\
\underline d&=\frac{d_2-d_1}{d_1+d_2}.  
\end{align*}
Clearly, 
$$\lim_{d\rightarrow\infty} \underline{d}=0.$$  
Now, increase the length of $W$ by appending $h=cd^\beta$ iid equiprobable 
symbols (from the same alphabet as that of $V$ and $W$) to it.   
Let $\Delta$ denote the size of the increase 
in the LCS score due to appending these $cd^\beta$ letters, i.e.,  
$$\Delta:=|LCS(V;W_1W_2\ldots W_d\ldots W_{d_2+cd^\beta-1}W_{d_2+cd^\beta})|-|
LCS(V;W)|.$$
First, for $d$ large,   
$\E|LCS(V,W)|$, is approximately equal to $\gamma_k(0)d$.  
Second, and as explain next, the expected gain $\E\Delta$ is approximately  
${cd^\beta}(\gamma_k(0) + \xi(d)\gamma_k^\prime(0^+))/2$, 
where $\xi(d)$ can take any value between $-1$ and $1$.  Indeed, by the very 
definition of $\gamma_k(\cdot,\cdot)$, 
$$\E\Delta=
\gamma_k\left(\bar{d}+\frac{cd^\beta}2,\frac{cd^\beta/2}{\bar{d}+
(cd^\beta/2)}+\frac{d_2-d_1}{d_1+d_2+cd^\beta}\right)
\left(\bar{d}+\frac{cd^\beta}2\right) - \gamma_k(\bar{d},\underline{d})\bar{d},$$
which, with the help of \eqref{alexander} and since, 
$$\frac{d_2-d_1}{d_1+d_2+cd^\beta}-\underline{d}=-\frac{\underline{d}cd^\beta/2}{\bar{d}+(cd^\beta/2)},$$
becomes 
$$\E\Delta=\gamma_k\left(\frac{cd^\beta/2}{\bar{d}+(cd^\beta/2)}+
\frac{d_2-d_1}{d_1+d_2+cd^\beta}\right) 
\left(\bar{d}+\frac{cd^\beta}2\right) - \gamma_k(\underline{d})\bar{d}+
O\left(\sqrt{d\ln d}\,\right),$$
i.e., 
\begin{equation}
\E\Delta=
\left(\gamma_k\!\left(\frac{(1-\underline{d})cd^\beta/2}{\bar{d}+(cd^\beta/2)}+\underline{d}\right)
-\gamma_k(\underline{d})\right)\!\!\left(\bar{d}+\frac{cd^\beta}{2}\right) + \gamma_k(\underline{d})\!\left(\bar{d}+\frac{cd^{\beta}}2-\bar{d}\right) 
+O\left(\sqrt{d\ln d}\right),\nonumber
\end{equation}
i.e., 
\begin{equation}
\label{alexander2}
\E\Delta=\gamma_k(\underline{d})\frac{cd^{\beta}}2+
\frac{\gamma_k\left(\frac{(1-\underline{d})cd^\beta/2}{\bar{d}+(cd^\beta/2)}+\underline{d}\right)
-\gamma_k(\underline{d})}{\frac{(1-\underline{d})cd^\beta/2}{\bar{d}+(cd^\beta/2)}}\,\frac{(1-\underline{d})cd^\beta}{2}
+O\left(\sqrt{d\ln d}\right), 
\end{equation}
above the idea can be informally 
summarized as:  $\delta(d\gamma_k)\approx d\delta(\gamma_k) + \gamma_k\delta(d)$.
Now, $d^\beta/(\bar{d}+(cd^\beta/2)) \to 0$, as $d\rightarrow\infty$ and $\gamma_k$ is concave, therefore  
$$\gamma_k^\prime(0^-)\!\geq\limsup_{d\to+\infty}
\frac{\gamma_k\!\left(\frac{(1-\underline{d})cd^\beta/2}{\bar{d}+(cd^\beta/2)}+\underline{d}\right)-\gamma_k(\underline{d})}
{\frac{(1-\underline{d})cd^\beta/2}{\bar{d}+(cd^\beta/2)}}
\!\geq\!\liminf_{d\to+\infty}
\frac{\gamma_k\!\left(\frac{(1-\underline{d})cd^\beta/2}{\bar{d}+(cd^\beta/2)}+\underline{d}\right)-\gamma_k(\underline{d})}
{\frac{(1-\underline{d})cd^\beta/2}{\bar{d}+(cd^\beta/2)}}\!\geq\!\gamma_k^\prime(0^+).$$
Hence, 
\begin{equation}
\label{samuel}
\gamma_k^\prime(0^-)\frac{cd^\beta}{2} +
o(d^\beta)\geq
\frac{\gamma_k\left(\frac{(1-\underline{d})cd^\beta/2}{\bar{d}+(cd^\beta/2)}+\underline{d}\right)-\gamma_k(\underline{d})}
{\frac{(1-\underline{d})cd^\beta/2}{\bar{d}+(cd^\beta/2)}}\, \frac{cd^\beta}2\geq 
\gamma_k^\prime(0^+)\frac{cd^\beta}{2}+
o(d^\beta).
\end{equation}
Now, using \eqref{samuel} with \eqref{alexander2} yields the desired order 
of magnitude for the expected gain:  
\begin{equation}
\label{orderexpectDelta2}
(\gamma_k(0)+\gamma_k^\prime(0^-))\frac{cd^\beta}{2}+o(d^\beta)\geq
\E\Delta\geq (\gamma_k(0)+\gamma_k^\prime(0^+))\frac{cd^\beta}{2}+o(d^\beta).
\end{equation}
In particular, when $\gamma_k^\prime(0^+)=\gamma_k^\prime(0^-)=0$, the above inequality becomes
\begin{equation}
\label{orderexpectDelta}
\E\Delta= \gamma_k(0)\frac{cd^\beta}{2}+o(d^\beta).
\end{equation}

As shown next, the order of magnitude of $\Delta$ 
is, with high probability, the same as the order of its expectation.  Indeed, 
in our context, the random variable $\Delta$ is a function 
of the iid entries $V_1V_2\ldots V_d$ and $W_1W_2\ldots W_{d+cd^\beta}$.   
Changing only one of its entries, changes $\Delta$ by at most $2$ and so 
by Hoeffding's martingale inequality and setting $u=2d+cd^\beta$,  
\begin{equation}
\label{inequ}
\P(|\Delta-\E\Delta|\geq t)\leq
2\exp\left( -\frac{2t^2}{u}\right),
\end{equation}
for all $t>0$.  Moreover, integrating out \eqref{inequ} gives  
$$\var\Delta\le u=2d+cd^\beta.$$

\paragraph{II) On the proportion of gaps aligned with the long constant block. }  
The previous arguments can now be used to understand 
the gaps aligned with the long block by an optimal alignment.  
Unlike in Part I), we consider here string $X$ and $Y$  
of length $2d$, which are iid 
except for $X$ containing  a long constant block in its middle.  
Again, $\pi$ is an optimal alignment of $X=X^aBX^c$ and $Y=Y^aY^bY^c$  
aligning $X^a$ with $Y^a$, $B$ with $Y^b$ 
and $X^c$ with $Y^c$.   Then, $\pi$ is modified to get 
$\bar{\pi}$ which aligns all 
of $B$ with gaps and $X^a$ with $Y^aY^b$.  
Again, the length of $B$ is $d^\beta$, with $1/2<\beta<1$, while 
the length of $Y^b$ is $h=cd^\beta$.  

Now, a somewhat over-simplified summary of the details 
of Section~\ref{proofs} is presented.  Assuming  $\gamma_k$ satisfies, say, Condition 2,  
it is shown in Section~\ref{proofs}) that with high probability $Y^a$ and $Y^c$  
have approximately the same length as $X^a$ and $X^c$. In other words, 
with high probability, the four strings $X^a$, $X^c$, $Y^a$ and $Y^c$ 
have length  $d+o(d)$.  So, the result of Part I applies to 
$X^a$ and $Y^a$, implying that, with high probability,  
$$|LCS(X^a,Y^aY^b)|-|LCS(X^a,Y^a)|=cd^\beta\frac{\gamma_k(0)}{2}+o(d^\beta).$$ 
Now, in the new alignment 
$\bar{\pi}$, the letters of the long block which 
were not aligned with gaps but with symbols from $Y^b$ are lost, 
and the loss is approximately ${h}/{k}={cd^\beta}/{k}$.  Moreover as mentioned earlier, 
the difference of the scores between $\pi$ and $\bar{\pi}$ is made up 
of the increase due to ``adding $Y^b$ to the alignment 
of $X^a$ and $Y^a$'' minus the loss in letters from the block $B$.  
Assuming $\gamma_k^\prime(0)=0$ and using \eqref{orderexpectDelta}, 
this difference is equal to: 
\begin{align*}&{\tt score\;of\;} \bar{\pi}-{\tt score\;of\;}\pi=\\ 
&(|LCS(X^a,Y^aY^b)|-|LCS(X^a,Y^a)|)\;\;-\;\;|LCS(B,Y^b)|=
c d^\beta\left(\frac{\gamma_k(0)}2 -\frac 1k\right)+o(d^\beta).
\end{align*}
Hence, whenever $\gamma_k(0)/2> {1}/{k}$, the change from $\pi$ to $\bar\pi$ typically increases 
the number of aligned letters and therefore $\pi$ cannot be an optimal alignment. 
In that case, the long constant block cannot be aligned with a piece $Y^b$ 
whose length-order is linear order $d^\beta$.  
In other words, the long constant block 
is, with high probability, mainly aligned with gaps.  

On the other hand, when ${\gamma_k(0)}/{2}<1/k$, then the score of $\pi$ is larger 
than the score of $\bar{\pi}$. 
So, an alignment like $\bar{\pi}$ cannot be optimal in that case.  In other words, 
when ${\gamma_k(0)}/{2}<1/k$ then, with high probability, any optimal alignment 
aligns most letters of the long block $B$ with letters and not with gaps.  
Here ``most letters'' indicates that at most $o(d^\beta)$ letters from the long block could  
get aligned with gaps.   These results are explained next, assuming $\gamma_k$ strictly concave 
in a neighborhood of the origin and having a derivative (equal to zero) at the origin.  
The same arguments, with minor changes, work 
as well without the strict concavity, assuming only that $\gamma_k^\prime(p_M)$ 
exists (and is therefore equal to zero).  
Using, in our developments, \eqref{orderexpectDelta2} rather than 
\eqref{orderexpectDelta}, 
the weaker condition: 
$$\left|\frac{\gamma_k^\prime(p_M^+)}{2}\right|< \left|\frac{\gamma_k(0)}{2}-\frac{1}{k}\right|
$$
will also do.  The difference in score between the alignment $\pi$ and $\bar{\pi}$ 
is then 
\begin{align*}&{\tt score\;of\;} \bar{\pi}-{\tt score\;of\;}\pi=\\
&|LCS(X^a\!,Y^aY^b)|\!-\!|LCS(X^a\!,Y^a)|\!-\!|LCS(B,Y^b)\!|=
\!cd^\beta\!\!\left(\!\!\frac{\gamma_k(0)}2\! +\xi(d)\frac{\gamma_k^\prime(p_M^+)}{2} -\frac 1k\!\right)\!\!+\!o(d^\beta), 
\end{align*}
where $\xi(d)$ can take any value between $-1$ and $1$.  
\paragraph{III) For which $k$ do we have  $k\gamma_k(0)> 2$?}
Whether a long constant block 
gets mainly aligned with gaps or not depends on $\gamma_k(0)$ being smaller or larger than $1/k$. 
It turns out, that $\gamma_k(0)$ is smaller than $2/k$ only for binary strings, 
that is when $k=2$.  For every $k\geq 3$, the opposite is true.  
Despite the exact values of $\gamma_k(0)$ not being known, 
there are rigorous bounds available, precise enough to show our assertions.  Anyhow, 
for large $k$, Kiwi, Loebl and Matou{\v{s}}ek \cite{kiwi} have shown that 
$\gamma_k(0)$ is of linear in $1/\sqrt{k}$, 
making $\gamma_k(0)/2$ strictly larger than $1/k$ 
when $k$ is large enough.  The case $k=3$ is near critical as can also be seen in our simulations.  
Taking the value of $0.717$ in Table~\eqref{table}, 
then $ \gamma_3(0)/2=0.3585$, which is slightly larger than 
$1/3$, specially since the order of magnitude of the precision by which 
the values $\gamma_k(0)$ are known is around $0.01$.

\section{Statements of Results}\label{statements}
In this section, we precisely state results indicating that 
the differentiability of the function $\gamma_k$ at $p_M$ 
controls the proportion of symbols from the long constant block which get 
aligned with gaps.  A kind of zero-one law holds true depending on the size of the alphabet.     

Below, both sequences $X$ and $Y$ 
have length $2d$, while the long block has length approximately equal to  
$d^\beta$, with $1/2 <\beta<1$, with $d$ large enough.  

We start with an example explaining how the aligned gaps are counted.  
For this, let $x:=00011100$ and let $y:=00011001$.  
The  first block of $x$ consists of three zeros, 
its second block consists of three ones, the third block consists of two zeros 
and the LCS of $x$ and $y$ is
$$LCS(x;y)=0001100,$$ 
which corresponds to the alignment 
$$
\begin{array}{l|l}
x&000111 00\\\hline
y&00011\;\;001\\\hline
LCS&00011\;\;00 
\end{array}
$$ 
In this alignment, the first  block of $x$ is only aligned with symbols, 
the second is aligned with one 
gap and so $1/3$ of its symbols gets aligned with gaps, 
while the last block of $x$ is only aligned with symbols and so 
the proportion of its symbols aligned with gaps is zero.

Let us next present two more examples to illustrate how with two 
letters, long constant blocks tend to be aligned with a proportion of gaps 
close to zero, while with three and more letters the opposite is true:

For this consider first the two binary strings:
$x=100101111{\bf 00000}101101101$ and $y=01111001011011011101001$.  
The alignment corresponding to the LCS is 
$$
\begin{array}{l|l}
x& 100101111{\bf 00\;\;0\;\;\;0\;\;\;\;0}10110110\;\;1\\\hline
y& \;\;0\;\;1\;\;111\;\;00101101101\;\;1101\;\;001
\end{array}
$$
Above, every $0$ 
from the long block is aligned with a $0$.  Let us next consider an example with 
six letters, and let $x=65 324 214 444 41  235 6631$ 
and $y=55  425 153 112 422 255 656$.  
The strings $x$ and $y$ in the previous example 
are ``generated" in the following way: 
Roll a fair six-sided die independently to obtain the strings 
everywhere except in the location of the long block.  
For the long block, i.e., for the  piece 
$x_{8}x_9x_{10}x_{11}x_{12}$, decide in advance 
to artificially introduce a long constant block: 
$x_8=x_9=\cdots=x_{12}$.  Outside that piece, roll the six-sided 
die independently, hence, 
for $x_1x_2x_3x_4x_5x_6x_7$, 
the die is rolled independently seven times, eight times for 
$x_{13}x_{14}x_{15}x_{16}x_{17}x_{18}x_{19} x_{20}$ and similarly fifteen times 
for the whole string $y$.  The alignment corresponding to the LCS $542112566$ is 
$$
\begin{array}{l|r}
x& \;\;   6   5\;\;32   42\;\;   1\;\;\;\;\;\;{\bf 44444}123\;\;\;\;\;\;\;\;\;    
56\;\;63 1\\ \hline
y&\;\;55\;\;\;\;425      1531{\bf \;\;\;\;\;\;\;\;\;\;\;}12\;\;422255656
\;\;\;\;
\end{array}
$$
The long block of five fours in $x$ is solely aligned with gaps.  

At this stage, formally describe the model with one inserted long constant block and sequences 
of length $2d$.  Let $X=X_1X_2\ldots X_{2d}$ and 
$Y=Y_1Y_2\ldots Y_{2d}$ be two independent strings of length 
$2d$.  A long constant block of length $\l$ is artificially inserted in 
the middle of the string $X$, replacing an iid part of equal length.  
Thus (assuming $\l$ even), 
$$\P\left(X_{d-(\l/2)+1}=X_{d-(\l/2)+2}=\ldots=X_{d+(\l/2)-1}=X_{d+(\l/2)}\right)=1.$$
while the rest the strings are iid with $k$ equiprobable letters. 
(Hence, $Y$, $X_1X_2\ldots X_{d-(\l/2)}$ and 
$X_{d+(\l/2)+1}X_{d+(\l/2)+2}\ldots X_{2d-1}X_{2d}$ are three 
independent iid strings with $\P(X_i=j)=\P(Y_i=j)=(1/k)$, $j=1,2,\ldots,k$, $i=1,2,\ldots,2d$.)

Next, let $\beta$ and $\alpha$ be constants 
independent of $d$ and such that 
\begin{equation} 
\label{beta1beta}
\frac12<\alpha<\beta<1,
\end{equation}
and let the length of the long constant block be $\l=d^\beta$.   
To formulate our first main result we further 
need two definitions:

{\bf Let $E^d$ be the event that the long constant block is mainly aligned 
with gaps}.  More precisely, $E^d$ is the event that at most $d^{\alpha}-1$ 
symbols of the long block get aligned with letters in any LCS-alignment.  In other words,   
the  score does not decrease by more than $d^{\alpha}-1$, when cutting 
out the long block:
$$
E^d:=\left\{|LCS(X_1X_2\ldots X_{d-(\l/2)-1}X_{d-(\l/2)}X_{d+(\l/2)+1}\ldots 
X_{2d};Y)| 
+d^{\alpha}> LCS(X;Y)\right\}.$$

{\bf Let $K^d$ be the event that replacing the long constant block with iid 
symbols approximately increases the LCS length by $\gamma^*_k/2$ times 
the length of the long constant block.}   Formally, 
let $\gamma_k^a$ be any constant, independent of $d$, and strictly smaller than 
$\gamma^*_k$, then $K^d$ is the event 
that when replacing the long constant block with iid symbols the length of the LCS 
increases by at least $(\gamma_k^a/2)d^\beta-d^{\alpha}$:
$$K^d:=\left\{|LCS(X^*;Y)| 
-|LCS(X;Y)|\geq \frac{\gamma_k^a}2 d^\beta-d^{\alpha}\right\},$$
where $X^*$ denotes the string obtained from $X$ by replacing
the long constant block by iid symbols.  In other words, for 
$i\in[1,d-(\l/2)]\cup 
[d+(\l/2)+1,2d]$, $X_i^*:=X_i$.  Moreover, the whole string   
$X^*=X^*_1X_2^*\ldots X^*_{2d}$ is iid.

We are now ready to formulate our main result for three or more letters.  

\begin{theorem}
\label{maintheoremsubsection}
Let
$k\gamma^*_k>{2}$,
and let also the mean
LCS function $\gamma_k: 
(-1,1)\rightarrow\mathbb{R}$, be differentiable at $p_M$.  
Let $1/2 < \alpha<\beta <1$.  Then, 
there exist constants $C_E>0, C_K>0$, independent of $d$,  such that 
$$\P(E^d)\geq 1-e^{-C_Ed^{2\alpha-1}},$$
and 
$$\P(K^d)\geq 1-e^{-C_Kd^{2\alpha-1}},$$
for all $d\geq 1$.
\end{theorem}

To give the result for the two-letter case, some more definitions are needed.  

\noindent
{\bf Let $G^d$ be the event that the long constant block gets mainly  
aligned with symbols and not with gaps.}  More precisely,  
$G^d$ is the event that the long constant block has 
(in any optimal alignment) at most 
$d^{\alpha}$ of its symbols aligned with gaps.   Equivalently,  
leaving out $d^{\alpha}$ symbols from the long constant block decreases 
the LCS by at least one unit.  Hence,   
$$G^d:=\left\{|LCS(X;Y)|>|LCS(X_1X_2\ldots X_{d-(\l/2)}X_{d-(\l/2)+d^{\alpha}+1}
X_{d-(\l/2)+d^{\alpha}+2}\ldots X_{2d};Y)|\right\}.$$
($X_1X_2\ldots X_{d-(\l/2)}X_{d-(\l/2)+d^{\alpha}+1} 
X_{d-(\l/2)+d^{\alpha}+2}\ldots X_{2d}$ is simply the string $X_1X_2\ldots X_{2d}$ from which the
piece $X_{d-(\l/2)+1}X_{d-(\l/2)+2}\ldots X_{d-(\l/2)+d^{\alpha}}$ has been removed.)

\paragraph{Let $H^d$ be the event that 
replacing the long constant block with iid symbols increases the LCS 
by at least ${\tilde c}_Hd^\beta$.} 
Here ${\tilde c}_H>0$ is any constant independent of $d$ and such that 
\begin{equation} 
\label{ch}{\tilde c}_H<\frac{3\gamma^*_2}2-1, 
\end{equation}
and so 
$$H^d:=\left\{|LCS(X^*;Y)|
-|LCS(X;Y)|\geq {\tilde c}_Hd^\beta\right\}.
$$

Let us next formulate our second main result for the two-letter case.

\begin{theorem}
\label{maintheoremsubsection2}
Let
$k\gamma^*_k < 2$,
and let the mean LCS function
$\gamma_2: (-1,1)\rightarrow\mathbb{R}$, 
be differentiable at $p_M$.  Then, 
there exist constants $C_G>0,C_H>0$, independent of $d$,  such that 
$$\P(G^d)\geq 1-e^{-C_Gd^{2\alpha-1}},$$
and 
$$\P(H^d)\geq 1-e^{-C_Hd^{2\alpha-1}},$$
for all $d\geq 1$.
\end{theorem}

The situation encountered for two letters 
might seem counter-intuitive at first.  Let us explain why: 
Consider for this two binary sequences of length $n$ where 
one string is made out only of ones while the other is 
made out of equiprobable zeros and ones.  Then the length of the LCS is the number of 
ones in the sequence 
with both symbols.  Since both symbols have probability $1/2$, the length of the 
LCS is approximately $1/2$ times the length of the strings.  However,  
for two binary iid sequences, the average length of the LCS is about 
$0.8$ times the length.  Hence, the LCS is much greater for two iid sequences, 
than when one sequence is made up of only one letter (i.e., one 
sequence is just ``a long constant block").  Thus, one would think 
that when within a sequence one gets an exceptionally long constant 
block, this should typically decrease the total LCS. 
Hence, since a long constant block  ``scores" much less than a typical 
piece of string iid drawn, one would expect that the long constant block tends 
to be ``left out" and not used too much (and hence tends to be aligned with many gaps).  
But the opposite is true!  Also, in optimal alignment,  
similar strings tend to be matched. Since a long constant block, 
is very different from an iid string, it thus would 
seem that a long block should be ``left out" and mainly 
matched with gaps.  This typically happens with three or more letters, 
but with two letters, the opposite is true.

Let us next further explain the binary situation on a illustrative example 
with two strings aligned in three different ways.  
Let $x=100101111 00000 101101101$
and $y=011110010 11011 011101001$ be two strings of total length 23 with 
$x$ containing a long constant block $00000$ of length $\l=5$.  
Consider now three alignments of $x$ and $y$. First, an alignment 
aligning the long block only with digits and with no gap:

{\footnotesize
$$
\begin{array}{c|c|c|c|c| c|c|c|c|c| c|c|c|c|c| c|c|c|c|c| c|c|c|c}
x& &1&0&0&1&0&1111& &00& &0&  &0&  & 0& &1&0&1101&1&0& &1\\\hline
y& & &0& &1& &111 & &00&1&0&11&0&11& 0& &1& &1101& &0&0&1
\end{array}
$$}

Here the long block $00000$ gets aligned with $0010110110$ 
having a length of $10$ which is 
twice the length of the long block. This is to be expected 
since the probability of $0$ is $1/2$, and so a string of length approximately $2\l$ is needed 
to get $\l$ zeros.  The above alignment can be viewed as consisting of 
three parts: the part to the left of the long block in $x$, the aligned long block, 
and the part to the right of the long block.  The part to the left  
aligns $5$ letter-pairs, the long block also gives $5$ letter-pairs and 
the piece to its right gives $7$ of them.   The total 
number of aligned letter-pairs in this alignment is thus $17$.  

Let us next try as second alignment, an alignment aligning the long block 
with a piece of string of similar size, e.g., of length $7$. 
For example, the alignment:  

{\footnotesize
\begin{equation}
\label{badalign}
\begin{array}
{c|c|c|c|c| c|c|c|c|c| c|c|c|c|c| c|c|c|c|c| c|c|c|c|c| c|c|c|c}
y& &0&1&  &111&0&   &0&1& &0  &11&0&11&  & & 0&1& &1&1&01& &0& & &0&1\\\hline
x& & &1&00& 1 &0&1 1& &1&1&000&  &0&  &  & & 0& & &1& &01&1&0&1&1&0&1
\end{array}
\end{equation}
}

This second alignment gives $4+3+6=13$ aligned letter-pairs which, as predicted, is a fewer  
number than the previous one:  Indeed, 
when aligning the long block entirely with letters and no gaps, 
the score tends to be higher.   
In the second alignment, the long block gets aligned with the piece of string $0110110$  
and with two of the zeros aligned with gaps.   Let us show, next, how to slightly modify 
this second alignment to provide a third alignment 
with an increased the total score.  
For this, take the two zeros from 
the long block which are aligned with gaps and align them with zeros from the 
string $y$.  To do so, take a piece of $y$ 
to the left of $y_9$ containing two zeros, i.e., take $y_6y_7y_8=001$.  Now align 
the two ``unused'' zeros, $x_{10} and x_{11}$ from the long block, with $y_6$ and $y_7$.  
Aligning the two  ``unused zeros'' leads to a score-gain of two, but at 
the same time to a loss, since previously  $y_6y_7y_8$ was aligned with $x_5x_6\ldots x_8=0111$.   
So, the previous alignment of $y_6y_7y_8$ with $x_5x_6x_7x_8$ 
has been destroyed creating a score-loss of two.  
However, $x_5x_6x_7x_8$ is now ``free,'' and so can be ``included'' into the alignment 
of $y_1y_2\ldots y_5$ with $x_1\ldots x_4$, meaning that $x_7x_8$ is aligned with $y_4y_5$.  
This addition gives a score-increase of two, and the total score-change is 
$2-2+2=2$.  Let us represent in ``toy'' form the three phases of the evolution between the 
second and third alignment (only 
the part of the alignment which is been modified is shown below): 
$$\begin{array}
{c|c|c|c|c| c|c|c|c|c| c|c|c }
y& &0&1&  &111&0&   &0&1& &  &\ldots\\\hline
x& & &1&00& 1 &0&1 1& &1&1&00&\ldots
\end{array}
$$
The first phase consists in aligning the two unused bits $x_{10}x_{11}$ 
from the long block with $y_7y_8$:  
$$\begin{array}
{c|c|c|c|c| c|c|c|c|c| c|c|c|c }
y& &0&1&  &111& &   & & & &00&1&\ldots\\\hline
x& & &1&00& 1 &0&1 1& &1&1&00& &\ldots
\end{array}
$$
This phase leads to a gain of two aligned letters since $x_{10}x_{11}$ gets aligned 
with $y_7y_8$, but at the same time to a loss of two aligned letter-pairs, 
since previously $x_5x_6x_7x_8x_9$ had two aligned letters and has none now. 
Next, ``bring the string $x_5x_6x_7x_8x_9=01111$ into the alignment":   
\begin{equation}
\label{smallalign}
\begin{array}
{c|c|c|c|c| c  }
y& &0&1&  &111\\\hline
x& & &1&00& 1 
\end{array}
\end{equation}
Now, in the alignment \eqref{smallalign},  
two ones on the right-end of $y$ are free (e.g., 
$y_4y_5$) providing, when aligned with two of the ones from 
$x_5x_6x_7x_8x_9$, two additional aligned letter-pairs with end result:  
$$\begin{array}
{c|c|c|c|c| c|c|c|c|c| c|c|c|c }
y& &0&1&  &1  & &11 & & & &00&1&\ldots\\\hline
x& & &1&00& 1 &0&1 1& &1&1&00& &\ldots
\end{array}
$$
The total score change is $2>0$ and therefore the alignment \eqref{badalign} 
cannot be optimal.  

In the second alignment, two zeros from the long constant 
block are not aligned with symbols.  Assume that instead of just $2$, we would have 
$j$, where $j$ is not too small.  Then, to align these $j$ zeros with 
zeros from the string $y$ would require a string of length approximately $2j$ 
in $y$ (in order to find $j$ zeros each having probability approximately $1/2$ of occurring). 
(Above, the piece of string from $y$ with which 
the free zeros from the long block were aligned was $y_6y_7y_8$ and had 
length $3$.)  Before changing the alignment, these $2j$ bits from $y$ were 
most likely aligned with about $2j$ bits from $x$. (Above, 
these bits are: $x_5x_6x_7x_8x_9$.)  When aligning these 
additional bits which became free, an approximate score-gain of 
$2j\gamma^*_2/2=j\gamma^*_2 \approx 0.8j$ 
is to be expected (with two sequences of length 
$j$ and so a total of $2j$ bits, the score is approximately equal 
to $j\gamma^*_2$).  Hence, the ratio score/bits is $\gamma^*_2/2$. 
(Using this average is a purely heuristic, since there is no proof that adding 
bits on one side of an alignment only produces an average 
increase of $\gamma^*_2/2$ per bit.)  Summing up:
\begin{itemize}
\item[a)] The new alignment of the $j$ free bits from the long block leads to a score-increase of $j$.
\item[b)] Undoing the previous alignment of the piece of string 
of $y$ which now gets aligned with the free 
bits of the long block, leads to an approximate loss of $2j\gamma^*_2$ 
bits since that piece has approximate length $2j$. 
\item[c)] Realigning the piece of $x$, which was previously aligned with the piece of $y$ 
getting now aligned with the free bits of the long block, leads to a score-gain.  
Since this last piece has an approximate length of $2j$, 
the score-gain is approximately $2j\gamma^*_2/2=j\gamma^*_2j$.  
\end{itemize}
Therefore, the total score-change is 
$$j-2j\gamma^*_2+j\gamma^*_2=j-j\gamma^*_2\approx 0.2j>0.$$

From the $2$-letter-strings examples presented above, 
the tendency is to align a long constant block with barely any gap. 
How much is then gained by replacing the long constant block by 
an iid piece?  Here is an heuristic answer: the long block gives 
$\l$ units, but uses a piece of length $2\l$ 
in the $y$-string.  This piece becoming 
free leads to a gain of $2\l$ bits plus the $\l$-bits from the long block. 
Hence using $3\l$ bits to get $\l$ points, and believing in 
the ``average point/bit number hypothesis'' for $\gamma^*_2/2$, 
lead to an approximate gain of $3\l\gamma^*_2/2$ aligned letter pairs. That 
is after replacing the long constant block by an iid piece, 
and realigning all the $3\l$ bits which were previously  
used with the long block.  Hence, approximately $3\l\gamma^*_2/2-\l\approx 0.215\l$ additional 
letter-pairs are available; for example, a long block of length 
$20$ would lead to an approximate average gain of $4$.  Extensive simulations,  
listed in the next section, demonstrate something very close.

\section{Simulations and the Nature of Alignments}\label{simulations}
It is very unlikely in an iid sequence of length $2d$ to find a constant block of length $d^\beta$.  
Typically, the blocks reach a length whose order is linear in $\ln d$.   
Nonetheless, our results proved for artificially inserted long blocks 
can be observed in simulations for naturally occurring block-lengths.  Our simulations 
are presented below.  

The first table gives estimates for the expected number of 
gaps in a block of length $\l$ placed in the middle of a string of length 
$1000$ (except for $\l>100$ where the string has length $4000$) 
as a function of $k$, the number of letters.  Since several 
optimal alignments might exist, we chose the one putting a maximum number of 
gaps into the long block.   Inserting a constant block with naturally-occurring length 
is similar to finding a constant block of that length in an iid sequence.  
Indeed, assume that there is a constant block of length $\l$, $\l$ not too small.  
Then, until such a block appears in an iid equiprobable binary sequence,  
it takes an expected $2^{\l}$ letters.  But, the contribution 
to the optimal alignment score of such a block would be at most $\l$, which is much smaller 
than the amount of symbols needed before encountering that block.  
So, heuristically, the constant block of length $\l$ has very little effect on the optimal alignment.  
Hence, the optimal alignment should more or less determine which parts get aligned with each other 
without regard to the long constant block.  This could then indicate that in terms of the number 
of gaps it gets aligned with, this long 
constant block behaves as if it had been artificially inserted.

{\footnotesize
$$
\begin{array}{c|c|c|c|c|c|c|c|c|c|c|c}
   &\l=1 &\l=2 &\l=5 &\l=10&\l=20&\l=30&\l=50&\l=100&\l=200&\l=300&\l=400\\\hline
k=2&0.53&1.67&2.25&2.75&4.2 &6.17&8.16&14.68&12.26&14.2 &19.6\\\hline
k=3&    &    &2.85&4.6 &12.5&18  &32.3&70.64&152.6&226  &\\\hline
k=4&0.72&1.19&3.27&6.78&16.3&25.6&43.8&88.4 &     &     &\\\hline
k=5&    &1.6 &3.36&7.76&16.3&27.1&49.7&96.2 &     &     &\\\hline
k=6&    &1.43&3.67&8.32&17.2&28.2&47.7&97.1 &     &     & \\\hline 
k=7&    &1.53&3.82&8.6 &18.7&27.9&48.6&98.1 &     &     & \\\hline
k=9&    &    &4.23&8.7 &18.4&29.2&48.4&     &     &     &
\end{array}
$$}

For each entry $100$ independent simulations are run.  
For each simulation, we find the number of 
gaps the block of length $\l$ gets aligned with and then compute the average of that number over the 
$100$ simulations.  This gives the entries 
of the above table.  The next table provides estimates for the ratio of the expected number 
of gaps and the length of the block.   Therefore, the next 
table is obtained from the previous one by dividing 
each entry by the value $\l$ corresponding to its column.  
The entries in the next table thus represent the 
``proportion of gaps'' in the long blocks depending on 
the length of the long block:

{\footnotesize
$$
\begin{array}{c|c|c|c|c|c|c|c|c|c|c|c}
   &\l=1 &\l=2 &\l=5 &\l=10&\l=20&\l=30&\l=50&\l=100&\l=200&\l=300&\l=400\\\hline
k=2&0.53&0.83&0.45&0.27&0.21&0.20&0.16&0.14 &0.06 &0.04 &0.04\\\hline
k=3&    &    &0.19&0.15&0.62&0.6 &0.64&0.7  &0.76 &0.75 & \\\hline
k=4&0.72&0.59&0.65&0.67&0.81&0.85&0.87&0.88 &     &     &\\\hline
k=5&    &0.8 &0.67&0.77&0.81&0.90&0.99&0.96 &     &     &\\\hline
k=6&    &0.7 &0.67&0.83&0.86&0.94&0.95&0.97 &     &     & \\\hline 
k=7&    &0.75&0.76&0.86&0.93&0.93&0.97&0.98 &     &     & \\\hline
k=9&    &    &0.8 &0.87&0.92&0.97&0.96&
\end{array}
$$
}

As seen above, with two letters, the proportion of gaps decreases as the length of the block 
increases, 
while for $k\geq 3$ the opposite is true ($k=3$ seems to be a close to the critical point, 
so this phenomenon kicks in only slowly).  Even for small block-length such as $\l=5$, this 
zero-one law seems to occur and, therefore, the micro-structure of the optimal alignment 
seems rather different for $k=2$ or $k\ge 3$.  

{\bf Which heuristic argument could explain that the result for artificially inserted 
long blocks result implies a similar one for iid sequences?}  
The simulations show that for naturally occurring long constant blocks, the phenomenon proved for 
artificially inserted ones continue to hold.  Now, for a block of length 
$\l_B$ much smaller than $d^\beta$ take the neighborhood 
of size $\l_B^{1/\beta}$ of that block.  In the optimal alignment 
of $X$ and $Y$, that neighborhood should also typically be aligned 
optimally.  So for that part of the alignment our 
results should apply.  Let us present an example: In the simulations 
when simulating the sequences $X$ and $Y$ of length 1000, 
replace in the sequence $X$ a piece of length $10$  
by a block of length $10$ somewhere in the middle of $X$, 
and then count the number of gaps it gets aligned with.  
An approach which would yield very similar results, 
would consist in finding the block of length $10$ 
closest to the middle of $X$ and then counting the number of 
gaps that block is aligned with in an optimal alignment.   
In simulations, by repeating these two operations a great number of times, 
in order to estimate the expected number of gaps 
a block of length $10$ gets aligned with, we find no significant 
difference between the two methods.   Heuristically, this lead to an important 
consequence:  Simulations seem to demonstrate that the results, on the proportion 
of aligned gaps in iid sequences with artificially inserted long blocks,   
for the naturally appearing long blocks appearing in an iid sequence continue to hold 
for the naturally appearing long blocks in an iid sequence.

Let us next display results giving the different numbers of gaps obtained 
at each simulation run. 
This should provide the reader with a sense for the order of the variance 
of the number of gaps in long blocks, when the length of the long block 
is held fixed.  Below $i$ is the result obtained 
with the $i$-th simulation.  Only blocks of length 
$\l=100$ are considered in the next table.

{\footnotesize
$$
\begin{array}{c|c|c|c|c| c|c|c|c|c|c}
   &i=1 &i=2 &i=3&i=4&i=5&i=6&i=7&i=8&i=9&i=10\\\hline
k=2&2   &1   &2  &31 &9  &0  &1  &3  &5  &7   \\\hline
k=3& 100&97  &30 &66 &76 &79 & 73&93 &74 &91 \\\hline
k=4&98  & 98 &99 &100&99 &93 &99&99  &100&60\\\hline
k=8&99  &100 &100&95 &100&99 &98&98  &100&99    \\\hline

\end{array}
$$
}

Let us further examine some of the entries in the table right above.  
For $k=4$ letters, two out of the ten simulations give $100$ gaps, 
four out of the ten give $99$ gaps, and once the much lower value of $60$ gaps.   
This seems to indicate that the number of gaps has a strongly 
skewed distribution.  Above the median estimate is $98.5$ which should be compared 
with the estimated expected number of gaps $88.4$ given in the first table.  
For the two-letter case, the respective estimates are $2.5$ and $14.68$. 
It thus appears that to take into account this skewness, a median estimation 
might be more appropriate than an expectation estimate and 
the discrepancy between the two-letter situation and the situation 
with more letters becomes even more pronounced when looking at the median.

The entries in the next table give the difference between the length of the  
LCSs when replacing the long block 
with iid entries in sequences of length $2d=1000$.  Again, 
$\l$ is the block length and $k$ the number of letters.  For each entry 
$100$ simulation runs are averaged.  Here and below the results for the small values of 
$\ell$ are displayed to show the progression the behavior as $\l$ increases.   

{\footnotesize
$$
\begin{array}{c|c|c|c|c|c|c|c|c|c|c|c}
   &\l=1 &\l=2  &\l=5 &\l=10&\l=20&\l=30&\l=50&\l=100&\l=200&\l=300&\l=400\\\hline
k=2&-0.01&0.06&0.52&0.9 &2.88&4.7 &9.48&21.8 &44.3 &73.5 &88.5 \\\hline
k=3&    &    &0.45&1.36&4.55&7.62&14.5&32.8  &68.9 &     &      \\\hline
k=4&0.03&0.06&0.59&1.85&5.3 &8.86&14.32&31.4 &     &     &\\\hline
k=5&    &0.2 &0.58&1.78&4.88&7.81&14.18&29.9 &     &     &\\\hline
k=6&    &0.1 &0.53&1.86&4.42&7.7 &12.8 &27.9 &     &     & \\\hline 
k=7&    &0.13&0.7 &2.05&4.7&7.3  &13.1 &27.28&     &     &\\\hline
k=9&    &    &0.7 &1.85&4.33&7.26&11.6 &     &     &     &    
\end{array}
$$
}

The next tables display the values to expect, from 
our heuristic arguments, for the typical increase in LCS for long constant 
blocks and the values obtained through simulations.  To start, 
let $k=2$, in which case our predicted change 
in LCS due to the replacement of the long block of length $\l$ 
is $(\gamma^*_2/2)3\l-\l\approx 0.215\l$.  

\renewcommand{\arraystretch}{1.5}
{\footnotesize
$$
\begin{array}{c|c|c|c|c|c|c|c|c|c|c|c}
        &\l=1 &\l=2  &\l=5 &\l=10&\l=20&\l=30&\l=50&\l=100&\l=200&\l=300&\l=400\!\!\\\hline
0.215\l  &0.215&0.43&1.07&2.15&4.3 &6.45&10.7&21.5  &43   &64.5  &86\\
\hline
\widehat{\E}
\Delta&-0.01&0.06&0.52&0.9 &2.88&4.7 &9.48&21.8 &44.3 &73.5 &88.5 \\
\hline
\end{array}
$$
}
Let us next compare simulated values with predicted values 
for $4$ letters alphabet where an 
increase of $(\gamma^*_4/2)\l\approx 0.325\l$ is expected.

{\footnotesize
$$
\begin{array}{c|c|c|c|c|c|c|c|c|c|c|c}
      &\l=1  &\l=2  &\l=5 &\l=10&\l=20&\l=30&\l=50&\l=100&\l=200&\l=300&\l=400\\\hline
0.325\l&0.325&0.65 &1.62&3.25&6.5&9.7&16.2&32.5   &     &     &\\\hline
\widehat{\E} \Delta
      &0.03 &0.06&0.59&1.85&5.3 &8.86&14.32&31.4 &     &     &\\\hline
 
\end{array}
$$}
Comparisons of simulated values with predicted values 
in case of $5$ letters, where an   
increase of $(\gamma^*_5/2)\l\approx 0.305\l$ is expected, are displayed next.

{\footnotesize
$$
\begin{array}{c|c|c|c|c|c|c|c|c|c|c|c}
      &\l=1  &\l=2  &\l=5 &\l=10&\l=20&\l=30&\l=50&\l=100&\l=200&\l=300&\l=400\\\hline
0.305\l&     &0.61&1.52&3.05&6.1  &9.15&15.2 &30.5&     &&\\\hline
\widehat{\E}\Delta
      &     &0.2 &0.58&1.78&4.88 &7.81&14.18&29.9  &     &     &\\\hline
 
\end{array}
$$}
Finally, for the $7$-letter case the increase is expected to be 
$(\gamma^*_7/2)\l\approx 0.27\l$.

{\footnotesize 
$$
\begin{array}{c|c|c|c|c|c|c|c|c|c|c|c}
      &\l=1  &\l=2  &\l=5 &\l=10&\l=20&\l=30&\l=50&\l=100&\l=200&\l=300&\l=400\\\hline
0.27\l&     &0.54 &1.35&2.70&5.4 &8.10&13.5&27.00&     &     &
\\\hline
\widehat{\E}\Delta
   &        &0.13&0.7 &2.05&4.7&7.3  & 13.1 &27.28   &     &     &\\\hline
 
\end{array}
$$}
As seen above, with more letters, the approximation is already quite good 
for blocks of lesser size.

\section{The Proofs}\label{proofs}
Throughout this section we are in the setting of 
Section~\ref{mainideas}:   
$X$ and $Y$ are two independent random sequences of length 
$2d$, the string $Y$ is iid while the string $X$ has a long constant block of 
size $\l$ (even) in its middle:  
$$\P\left(X_{d-(\l/2)+1}=X_{d-(\l/2)+2}=\ldots=X_{d+(\l/2)-1}=X_{d+(\l/2)}\right)=1,$$
and is iid everywhere else.   
Moreover, the symbols are equally likely on an alphabet of size $k$.  

\subsection{Proofs For Three Letters or More}\label{rigorous3}

In this subsection, assume that 
\begin{equation}
\label{3ormore}
k\gamma^*_k>2. 
\end{equation} 
To start with, some heuristic arguments are given to explain why under 
the condition \eqref{3ormore}, and in any optimal alignment, 
the artificially inserted long constant block is mainly aligned with gaps (see also  
Part I and Part II in Section~\ref{mainideas}), the proofs then follow.    

As far as the heuristics is concerned, proceed by contradiction.    
Indeed, assume on the contrary that there is an optimal alignment $\pi$ 
with $m$ symbols from the long constant block aligned with symbols. 
Then, by equiprobability, in order to get 
$m$ times the same letter in a contiguous substring of $Y$, 
typically requires a piece of length approximately equal to $km$.  Therefore, 
if $m$ symbols from the long constant block get aligned with 
symbols, then typically a piece of $Y$ of length approximately equal to $km$ 
is required.  Next, modify the alignment $\pi$.  To do so, 
take the piece of $Y$ which was used for the $m$ symbols of the long 
constant block and align it otherwise.  Let $\bar\pi$ be 
the new alignment obtained in this way.   In this way, $m$ aligned letters from the 
long constant block are lost but 
realigning the $km$ symbols of $Y$ adds approximately $km(\gamma^*_k/2)$ aligned symbols elsewhere.    
So the change is approximately 
$$\frac{\gamma^*_k}{2}km-m=m\left(\frac{k\gamma^*_k}{2}-1\right).$$ 
But from \eqref{3ormore}, $k\gamma^*_k >2$, 
and so the change due to realigning the $km$ symbols from 
$Y$ outside the long block, typically leads to an increase in the 
number of aligned symbols.  Hence, $\pi$ aligns fewer letter-pairs than 
$\bar\pi$, and therefore $\pi$ cannot be an optimal alignment.

Let us now proceed to the formal arguments and to do so, 
recall that in Section~\ref{statements} we defined: 
\begin{itemize}
\item $E^d$, the event that the long constant block is mainly aligned 
with gaps: 
$$
E^d=\left\{\!|LCS(X_1X_2\ldots X_{d-(\l/2)-1}X_{d-(\l/2)}X_{d+(\l/2)+1}\ldots 
X_{2d};Y)| 
+d^{\alpha}\!> LCS(X;Y)\!\right\}.$$
\item $K^d$, the event that replacing the long constant block with iid 
symbols leads to an approximate length-increase of $\gamma^*_k/2$ times 
the length of the long constant block:
$$K^d=\left\{|LCS(X^*;Y)| 
-|LCS(X;Y)|\geq \frac{\gamma_k^a}2 d^\beta-d^{\alpha}\right\}.$$
\end{itemize}

We intend to prove that if $\gamma^*_k/2>1/k$, 
then both events $E^d$ and $K^d$ occur with high probability, i.e., we intend to prove  
Theorem~\ref{maintheoremsubsection}.  For this proceed as follows: 
First define four events $B^d$, $C^d$, $D^d$ and $F^d$ and in 
Lemma~\ref{combi1} prove that 
$$B^d\cap C^d\cap D^d\subset E^d,$$
why Lemma~\ref{lemmaB}, \ref{lemmaC} and \ref{lemmaD} respectively show that 
$B^d$, $C^d$ and $D^d$ occur with high probability and thus 
so does $E^d$.  Next, Lemma~\ref{combiO} show that 
$$B^d\cap C^d\cap D^d\cap F^d \subset K^d,$$
while Lemma~\ref{lemmaF} shows that 
$F^d$ occurs with high probability and, thus, so does $K^d$.  

Recall also that $\alpha$ and $\beta$ and reals independent of $d$, such that 
$1/2 <\alpha<\beta<1$; that $d^{\beta}$ is the length of the artificially  
inserted long constant block and that $d^{\alpha}$ is the maximum number of 
symbols, from the long constant block, which 
can get aligned with symbols instead of gaps.  
Finally, for $p\in (-1,1)$, recall the definitions of 
$\gamma_k(n,p)$ and $\gamma_k(p)$ as respectively given in  
\eqref{gammaknp} and \eqref{meancurve} and the definition of $p_M$ given towards the 
end of the introductory section.

Let us next introduce some more notations.  
     
\begin{itemize} 
\item Let $\kappa$, $\gamma_k^a$, $\gamma_k^b$ and $\gamma_k^c$ 
be constants, independent of $d$, such that 
\begin{equation}
\label{kk_1}
\frac{2}{k}<\frac{2}{\kappa}<{\gamma_k^a}<{\gamma_k^b}<\gamma_k^c<{\gamma^*_k}.
\end{equation}
One thinks of $\kappa$ as being approximately equal to $k$ 
while $\gamma_k^a$, $\gamma_k^b$ and $\gamma_k^c$ are all very close to $\gamma^*_k$.  
\item Let $q\in(0,1)$ be such that 
$$\gamma_k(-q)=\gamma_k(q)=\gamma_k^c,$$
with also $\gamma_k^\prime(q^+) = \gamma_k^\prime(q^-)$.  
(The concavity of $\gamma_k$, clearly ensures that such a $q$ exists and is also such that  
\begin{equation}
\label{q1q2c}
\forall r\in[-q,q],\quad \gamma_k(r)\geq \gamma_k^c.) 
\end{equation}
Assume also, for $k\ge2$, that for all $p_1,p_2\in[-q,q]$, 
\begin{equation}
\label{gamma'}
\left| \frac{\gamma_k(p_2)-\gamma_k(p_1)}{p_2-p_1}  \right|
< \frac{(\gamma^c_k-\gamma^b_k)}{16}\,, 
\end{equation}
Since the derivative at $p_M$ exists and 
is therefore zero, 
it is always possible to determine $\gamma_k^b$ and $\gamma_k^c$ 
so that \eqref{kk_1} and \eqref{gamma'} simultaneously hold. 
For this simply keep $\gamma^b_k$ fixed and let $\gamma_k^c$ 
converge from below to $\gamma^*_k$.  When $\gamma_k^c$ 
gets close enough to $\gamma_k(0)=\gamma^*_k$, then the conditions 
are fulfilled.  

In case $k=2$, assume further that $\gamma^{II}_2$ is such that  
\begin{equation}
\label{gamma'bis}
\gamma_2^\prime(q)\leq \frac{\gamma^{II}_2-\gamma^*_2}{16}, 
\end{equation}
and that $\tilde{\gamma_2}$ is such that 
\begin{equation}
\label{gamma'bisbis}
\frac{|\gamma_2(p_2) - \gamma_2(p_1)|}{|p_2-p_1|}\leq
\frac{\gamma_2^c-\tilde{\gamma_2}}{32},   
\end{equation}
for all $p_1, p_2 \in [-q,q]$.  
Both the above conditions are satisfied from our assumptions on the derivative of $\gamma_2$ (for example, 
take $\tilde{\gamma}_2$ close but smaller than $\gamma_2^*$.  
Then, let $q \to 0$ and take $\gamma_2^c$ closer and closer to $\gamma_2^*$ 
till \eqref{gamma'bisbis} is satisfied).

\item Let $i_1$ and $i_2$ be {\it the respective integer 
rounding of each right-hand side below}:    
\begin{equation}
\label{i1}
i_1:=\frac{1-q}{1+q}\left(d-\frac{\l}{2}   \right), 
\end{equation}
and
\begin{equation}
\label{i2}
i_2:=\frac{1+q}{1-q}\left(d-\frac{\l}{2}   \right),
\end{equation}
where again $\l:=d^\beta$ (is even) and $1/2<\beta<1$ does not depend on $d$.  
Clearly, both $i_1$ and $i_2$ both depend on $d$. 
Moreover, whenever $i\in [i_1,i_2]$, then  
$$\frac{\E|LCS(X_1X_2\ldots X_{d-(\l/2)}; Y_1Y_2\ldots Y_i)|}{d^*}
= \gamma_k(d^*,r),$$
with $d^*=(d-(\l/2)+i)/2$ and $r=(i-d+\ell/2)/(i+d-\ell/2)$.  Since $r\in [0,q]$,  
if $i\ge d -\ell/2$ while $r\in[-q,0]$, if $i\le d -\ell/2$, it also follows that 
$\gamma_k(r) \geq \gamma_k^c$.
\end{itemize}

\paragraph{Let $B^d$ be the event that to find $d^{\alpha}$ times the same symbol 
in $Y$, a piece of length at least $h=\kappa d^{\alpha}$ is needed.}  \

More precisely, let $B^d(i,h)$ be the event that, in the string $Y_iY_{i+1}\ldots Y_{i+h}$, 
every letter appears at most $h/\kappa$ times. For this, 
let $r\in\{1,2,\ldots,k\}$ and 
let $W_j(r)$ be the Bernoulli random variable which is equal to one 
if $Y_j=r$ and zero otherwise.  With these notations, let 
$$B^d(i,h) := \left\{\forall r=1,2,\ldots,k: \sum_{j=i}^{i+h}W_j(r)\le \frac h{\kappa}\right\},$$ 
and let 
$$B^d:=\bigcap_{i\in[1,2d]}B^{d}(i,h).$$ 

\paragraph{Let $C^d$ be the event that for every $i\in[i_1, i_2]$, the length of the optimal 
alignment between $X_1X_2\ldots X_{d-(\l/2)}$ and $Y_1Y_2\ldots Y_{i+h}$, 
$h = \kappa d^{\alpha}$, is larger, by at least $h\gamma^a_k/2$, than the length of the optimal 
alignment between $X_1\ldots X_{d-(\l/2)}$ and $Y_1Y_2\ldots Y_{i}$.} \

More precisely,
let $C_R^d(i,h)$ be the event that by concatenating $h$ letters to the right 
of $Y_1Y_2\ldots Y_i$, the LCS with $X_1X_2\ldots X_{d-(\l/2)}$ 
increases by at least $h\gamma^a_k/2$.  Hence the event 
$C_R^d(i,h)$ holds when 
$$|LCS(X_1X_2\ldots X_{d-(\l/2)};Y_1Y_2\ldots Y_{i+h})| 
-|LCS(X_1X_2\ldots X_{d-(\l/2)};Y_1Y_2\ldots Y_i)|\geq \frac{\gamma^a_kh}2,$$ 
and  
\begin{equation}\label{defC}
C_R^d:=\bigcap_{i\in[i_1,i_2]}C_R^d(i,h),
\end{equation}
where $i_1$ and $i_2$ are defined in \eqref{i1}.  In a similar fashion, 
define $C_L^d(i,h)$ to be the event that by concatenating $h$ letters to the left 
of $Y_1Y_2\ldots Y_i$, the LCS with $X_1X_2\ldots X_{d-(\l/2)}$
increases by at least $h\gamma^a_k/2$, and then, as above, one defines $C_L^d$.  
Finally, let $C^d = C_R^d \cap C_L^d$.

\paragraph{Let $D^d$ be the event that any optimal alignment aligns 
$X_{d-(\l/2)}$ into the interval $[i_1,i_2]$.} \ 

To define $D^d$ precisely, let us first set a convention: 
when an alignment $\pi$ aligns $X_i$ with $Y_j$, then 
$X_i$ is said to be aligned with $j$ under $\pi$. 
If $\pi$ aligns $X_i$ with a gap, then let 
$i_I$ be the largest $m < i$ such that $X_m$ gets aligned with a symbol and 
not with a gap.  If $X_{i_I}$ gets 
aligned with $Y_j$, then $X_{i_I}$ is said to be aligned with $j$ under $\pi$. 

Next, let $D^d_I$ be the event that for any optimal alignment 
$\pi$ of $X$ and $Y$, if $i$ is the spot where $\pi$ 
aligns $X_{d-(\l/2)}$, then $i\in[i_1,i_2]$.  
Similarly, let $D^d_{II}$ be the event that for any optimal alignment 
$\pi$ of $X$ and $Y$, if $i$ designates the spot where $\pi$ aligns 
$X_{d-(\l/2)+\kappa d^{\alpha}}$, then $i\in[i_1,i_2]$.  Finally, let  
$$D^d:=D^d_I\cap D^d_{II}.$$

\paragraph{Let $F^d$ be the event that for every $i\in[i_1, i_2]$, the length of the optimal 
alignment between $X_1^*X_2^*\ldots X_{d-(\l/2)+d^{\beta}}^*$ and $Y_1Y_2\ldots Y_{i}$ 
is larger, by at least $d^\beta\gamma^a_k/2$, than the length of the optimal 
alignment between $X_1^*X_2^*\ldots X_{d-(\l/2)}^*$ and $Y_1Y_2\ldots Y_{i}$.}  \ 

More precisely, 
$$F^d := \bigcap_{i\in[i_1,i_2]}F^d_i,$$
where 
\begin{align*}
F^d_i:=\{|LCS(X_1^*X_2^*\ldots X_{d-(\l/2)+d^\beta}^*&;Y_1Y_2\ldots Y_i)| - \\  
&  |LCS(X_1^*X_2^*\ldots X_{d-(\l/2)}^*;Y_1Y_2\ldots Y_i)| \ge \frac{\gamma^a_kd^\beta}{2} \}.
\end{align*}

We now prove the first combinatorial lemma of this subsection:

\begin{lemma}
\label{combi1}
$$B^d\cap C^d\cap D^d\subset E^d.$$
\end{lemma}

\begin{proof} The proof is by contradiction and so assume that $E^d$ does not hold. 
Then, there is
an optimal alignment $\pi$ for which there are at least $d^{\alpha}$ 
letters, from the long constant block, which are not aligned with gaps.  
Moreover, without loss of generality, assume that these 
letters are at the beginning of the block.  
However, when the event $B^d$ holds true, the first $d^{\alpha}$ letters 
from the long block are aligned with a portion of $Y_1Y_2\ldots Y_{2d}$ 
of length at least $\kappa d^{\alpha}$.  In other words, there exists 
$i\in[1,2d]$ such that the optimal alignment $\pi$ 
aligns the first $d^{\alpha}$ letters from the long block with 
$Y_{i+1}\ldots Y_{i+t}$, where $t\geq \kappa d^{\alpha}$.  Therefore, the optimal alignment 
$\pi$ only aligns $Y_{i+1}\ldots Y_{i+t}$ with those first $d^{\alpha}$ 
letters from the long block.  Next, since $D^d$ holds true, 
assume that $i\in [i_1,i_2]$.   Now, modify the alignment $\pi$ so as 
to no longer align these $d^{\alpha}$ letters from the long block 
with $Y_{i+1}\ldots Y_{i+t}$.   In doing so, $d^{\alpha}$ aligned letters are lost.  
In turn, $Y_{i+1}\ldots Y_{i+t}$ can be realigned with a part of $X$ 
outside the long block.  In other words, align now $Y_1Y_2\ldots Y_{i+t}$ 
entirely with $X_1X_2\ldots X_{d-(\l/2)}$.   But, the event $C^d$ guarantees, 
since $t\geq \kappa d^{\alpha}$, a gain of at least $\kappa d^{\alpha}\gamma_k^a/2$. 
Summing up the 
losses and the gains, obtained in modifying $\pi$, lead to an increase 
of at least
\begin{equation}
\label{increasea}
\frac{\kappa d^{\alpha}\gamma_k^a}{2}
-d^{\alpha}=d^{\alpha}\left(\frac{\kappa\gamma_k^a}{2}-1  \right) > 0, 
\end{equation}
by \eqref{kk_1}.  Therefore, 
$\pi$ is not optimal which is a contradiction. 
\end{proof}

\

Let us now state and prove a second combinatorial lemma recalling that $X^*$ 
denotes the string obtained from $X$ by replacing the long constant block by iid symbols.

\begin{lemma}
\label{combiO}
$$B^d\cap C^d\cap D^d\cap F^d\subset K^d.$$
\end{lemma}

\begin{proof} By the previous lemma, when the events $B^d$, $C^d$ and $D^d$ hold true, 
any optimal alignment aligns at most $d^{\alpha}$ letters, from the long block, with 
letters. Let $\pi$ be an optimal alignment of $X$ and $Y$, then 
$\pi$ aligns at least $d^\beta-d^{\alpha}$ symbols from the long block 
with gaps.  Assume that $X_{d-(\l/2)}$ gets aligned with $Y_i$ by $\pi$.  
Now, transform $\pi$ into a new alignment $\bar\pi$ aligning $X^*$ and $Y$ in the following manner:  
Instead of aligning $X_1^*X_2^*\ldots X_{d-(\l/2)}^*$ 
with $Y_1Y_2\ldots Y_i$, concatenate $X^*_{d-(\l/2)+1}X^*_{d-(\l/2)+2}\ldots X^*_{d+(\l/2)}$ 
to the $X$-part, and align $Y_1Y_2\ldots Y_i$ with 
$X_1^*X_2^*\ldots X_{d+(\l/2)}^*$ in an optimal way, i.e., in such a way that any chosen 
alignment corresponds to a LCS of $Y_1Y_2\ldots Y_i$ and
$X_1^*X_2^*\ldots X_{d+(\l/2)}^*$.  Next, for the remaining letters of the strings $X^*$ 
and $Y$, use the alignment $\pi$.  Hence, if $m\in[i+1,2d]$ and 
$n\in[d+(\l/2)+1,2d]$ and if $\pi$ aligns $X_n$ with $Y_m$, 
then $\bar\pi$ aligns $X_n^*$ with $Y_m$.  Since $D^d$ holds, $i\in[i_1,i_2]$ and since 
$F^d$ holds, concatenating 
$X^*_{d-(\l/2)+1}X^*_{d-(\l/2)+2}\ldots X^*_{d+(\l/2)}$ 
to the $X$-part leads to a score-increase of at least $(\gamma_k^a/2)d^\beta$.  
But the transformation of $\pi$ into $\bar\pi$ could decrease the score.  Indeed, 
up to $d^{\alpha}$ letters, from the long block, could under $\pi$  
have not been aligned with gaps  (and thus could have been aligned with letters).  
Therefore, replacing the long block 
by $X^*_{d-(\l/2)+1}X^*_{d-(\l/2)+2}\ldots X^*_{d+(\l/2)}$ might lead to 
a loss not exceeding $d^{\alpha}$ aligned letter-pairs.  Summing up the 
losses and the gains, obtained in modifying $\pi$, lead to an increase of at least 
$$\frac{\gamma_k^a}2 \; d^\beta-d^{\alpha}.$$
This proves that the event $K^d$ holds and finishes this proof. 
\end{proof}

\

Let us now show that $B^d$, $C^d$, $D^d$ and $F^d$ 
all occur with high probability.  For $C^d$, let us start with the following:

\begin{lemma}
\label{difficult}
Let $\l:=d^\beta$.  For all $d$ large enough 
and all $i\in[i_1,i_2]$,
\begin{equation}
\label{delta}
\E\!\left(|LCS(X_1^*X_2^*\ldots X^*_{d-(\l/2)};Y_1Y_2\ldots Y_{i+h})|
-|LCS(X_1^*X_2^*\ldots X^*_{d-(\l/2)};Y_1Y_2\ldots Y_i)|\right)\!\ge 
\frac{h\gamma^b_k}2, 
\end{equation}
where $h=\kappa d^{\alpha}$. 
\end{lemma}
\begin{proof}Let 
$$\Delta := |LCS(X_1^*X_2^*\ldots X^*_{d-(\l/2)};Y_1Y_2\ldots Y_{i+h})|
-|LCS(X_1^*X_2^*\ldots X^*_{d-(\l/2)};Y_1Y_2\ldots Y_i)|,$$
and assume at first that $i\in [i_1, d]$ (see \eqref{i1}).  
By definition, 
$$\E\Delta=d_2\gamma_k(d_2,p_2)-d_1\gamma_k(d_1,p_1),$$
where
\begin{align*}
d_2:=&\frac12 \left(i+d-\frac\l 2+h\right)\!,  \, \quad p_2:=\frac{i-d+(\l/2)+h}{i+d-(\l/2)+h},\\
d_1:=&\frac12 \left(i+d-\frac\l 2\right), \quad \quad \quad p_1:=\frac{i-d+(\l/2)}{i+d-(\l/2)}.
\end{align*}
From Alexander~\cite{Alexander2} (see also \eqref{alexander}), there exists a constant $C_\gamma>0$ 
(independent of $d$ and $p$) such that 
\begin{equation}
\label{sqrt(d)}
|\gamma_k(d,p)-\gamma_k(p)|\leq C_\gamma\sqrt{\frac{\ln d}{d}}\,,
\end{equation}
for all $p\in (-1,1)$ and all $d\ge 1$.  
Using \eqref{sqrt(d)} and since 
$d_1,d_2\leq 2d$, 
\begin{equation}
\label{Deltad2}
\E\Delta\geq d_2\gamma_k(p_2)-d_1\gamma_k(p_1)-
2C_\gamma\sqrt{2d \ln 2d}. 
\end{equation}
Now, 
\begin{equation}
\label{Deltad3}
d_2\gamma_k(p_2)-d_1\gamma_k(p_1)
=\frac{h\gamma_k(p_1)}2+d_2{\delta\gamma_k},  
\end{equation}
where 
$$\delta\gamma_k=\gamma_k(p_2)-\gamma_k(p_1).$$  
By the concavity and the symmetry of $\gamma_k$, if $p_2\leq 0$, and since $p_1 < p_2$, then 
$\delta\gamma_k\geq 0$ so that 
\begin{equation}
\label{I}
d_2\gamma_k(p_2)-d_1\gamma_k(p_1)\geq \frac{h\gamma_k(p_1)}2.
\end{equation}
If $p_2\geq 0$, then 
\begin{equation}
\label{delta_p}
\delta p := p_2 -p_1 =
\frac{2h(d-(\l/2))}{(i+d-(\l/2))(i+d-(\l/2)+h)}\le 2\frac{h}{d},
\end{equation}
for $d$ large enough, e.g., $i\ge \l/2$, i.e., $(1-q)d \ge \l$.  
Since $i\in[i_1,i_2]$ and $i+\kappa d^{\alpha}\in[i_1,i_2]$, then 
$p_1,p_2\in [-q,q]$ and \eqref{gamma'} lead to 
\begin{equation}
\label{aaaa}
\frac{|\delta\gamma_k|}{\delta p} < 
\frac{(\gamma^c_k-\gamma^b_k)}{16}.
\end{equation}
Combining \eqref{aaaa} with \eqref{delta_p} and since $d_2\leq 2d$, 
\begin{equation}
\label{additional}
\frac{h\gamma_k(p_1)}2+d_2\,\frac{|\delta\gamma_k|}{\delta p}\,\delta p 
\ge h\left(\frac{\gamma_k(p_1)}2-\frac{(\gamma^c_k-\gamma^b_k)}{4}\right).  
\end{equation}
Now, by the very definition of $i_1,i_2$, 
$\gamma_k(p_1)\geq \gamma_k^c$, which yields 
\begin{equation}\label{eq.xx}
h\left(\frac{\gamma_k(p_1)}2-\frac{(\gamma^c_k-\gamma^b_k)}{4}\right)
\ge 
h\left(\frac{\gamma_k^c}2-\frac{(\gamma^c_k-\gamma^b_k)}{4}\right)= h\left(\frac{\gamma^b_k}{2}+ 
\frac{(\gamma^c_k-\gamma^b_k)}{4}\right).\end{equation}
Next, \eqref{eq.xx} together with \eqref{additional}, 
\eqref{Deltad2} and \eqref{Deltad3} lead to:   
\begin{equation}
\label{Deltad4}
\E\Delta\geq \kappa d^{\alpha}\frac{\gamma^b_k}{2}+\kappa d^\alpha
\left(\frac{\gamma_k^c-\gamma_k^b}{4}\right)-2C_\gamma\sqrt{2d \ln 2d}.
\end{equation}
Finally, since $\alpha>1/2$ is independent of 
$d$, and since 
$\gamma_k^c - \gamma_k^b>0$, $2C_\gamma\sqrt{2d \ln 2d}$ becomes ``negligible" when compared 
to $\kappa d^{\alpha}(\gamma_k^c-\gamma_k^b)/4$.  So, for large enough 
$d$, \eqref{Deltad4} implies that 
$\E\Delta\geq \kappa d^{\alpha}{\gamma^b_k}/{2}$, 
which is what we intended to prove, at least for $i\in [i_1, d]$.  
As shown next, for $i\in [d, i_2]$  
and $d$ large enough, the inequality \eqref{delta} remains valid.  Indeed, 
at first, one only needs $i\in [i_1, d]$ instead of $i\in[i_1,i_2]$, to obtain \eqref{aaaa}; 
more specifically one needs $i+\kappa d^{\alpha}\in[i_1,i_2]$.  
However, \eqref{gamma'} is a strict inequality and so, for 
$d$ large enough, even if $i+\kappa d^{\alpha} \notin [i_1,i_2]$ but 
as long as $i\in[i_i,i_2]$, by continuity ($p_2-p_1 \le 2h/d$) and since 
$\gamma_k^\prime(q^+) = \gamma_k^\prime(q^-)$, 
the inequality \eqref{aaaa} still holds.  
This then implies \eqref{delta}.\end{proof} 

\

The next lemma shows that the event $C^d$ occurs with high probability:

\begin{lemma}
\label{lemmaC}
For $d$ large enough, 
$$\P(C^d)\geq 1- 2d\exp
\left(-\frac{d^{2\alpha-1}\kappa^2(\gamma^a_k-\gamma^b_k)^2}{18}\right).$$ 
\end{lemma}
\begin{proof} Let $i$ be a positive integer, let $h=\kappa d^{\alpha}$ and let 
$$\Delta:=|LCS(X_1^*X_2^*\ldots X^*_{d-(\l/2)};Y_1Y_2\ldots Y_{i+h})|
-|LCS(X_1^*X_2^*\ldots X^*_{d-(\l/2)};Y_1Y_2\ldots Y_i)|.$$
Recalling the very definition of the event $C_R^d(i,h)$ (see \eqref{defC}), note 
that when its complementary $(C_R^{d}(i,h))^c$ holds, then 
\begin{equation}
\label{Deltagamma}
\Delta\leq \frac{h\gamma^a_k}2. 
\end{equation}
Now, $\Delta$ is function of the iid entries 
$X^*_1,X^*_2,\ldots,X^*_{d-(\l/2)}$ and $Y_1,Y_2,\ldots,Y_{i+h}$ 
and changing one of them changes $\Delta$ 
by at most $2$.  Assuming that $i\in[i_1,d]$, Lemma~\ref{difficult} applies and 
$\E\Delta\geq \gamma_k^b h/2$.
Together with \eqref{Deltagamma}, this leads to 
\begin{equation}
\label{Deltagamma3}
\Delta-\E\Delta\leq 
\frac{h\gamma^a_k}2-\frac{h\gamma^b_k}2=\frac{h(\gamma^a_k-\gamma^b_k)}2.
\end{equation}
In other words, the event $(C_R^{d}(i,h))^c$ implies that 
the inequality \eqref{Deltagamma3} holds true.  Hence, 
\begin{equation}
\label{Deltagamma4}
1-\P(C_R^{d}(i,h))\leq
\P\left(\Delta-\E\Delta\leq d^*
\left(\frac h{d^*}\right)\frac{(\gamma^a_k-\gamma^b_k)}2\right)
\end{equation}
where $d^*:=d-(\l/2)+i+h$.  
By assumption, $\gamma^a_k-\gamma^b_k<0$ and therefore by Hoeffding's inequality, the right-hand 
side of \eqref{Deltagamma4} is upper bounded by 
\begin{equation}
\label{exp}\exp\left(-d^*\left(\frac h{d^*}\right)^2
\frac{(\gamma^a_k-\gamma^b_k)^2}2\right).
\end{equation}
Since $d^*\leq 3d$ and since, for $d$ large enough $d^*\geq d$, 
\eqref{Deltagamma4} becomes 
$$1-\P(C_R^{d}(i,h))\leq 
\exp\left(-d^{2\alpha-1}\kappa^2\frac{(\gamma^a_k-\gamma^b_k)^2}{18}\right),$$
with our choice of $h=\kappa d^\alpha$, $1/2<\alpha$.  
Hence, since the interval $[i_1,d]$ contains at most $d$ elements, 
$$1-\P(C_R^{d})\leq d\exp
\left(-d^{2\alpha-1}\kappa^2\frac{(\gamma^a_k-\gamma^b_k)^2}{18}\right).$$  
A symmetric argument leads to the same bound for 
$C_L^d$ and, since $C^d = C^d_R\cap C^d_L$, 
$$1-\P(C^{d})\leq 2d  \exp
\left(-d^{2\alpha-1}\kappa^2\frac{(\gamma^a_k-\gamma^b_k)^2}{18}\right).$$\end{proof}

Next, the event $B^d$ is shown to hold with high probability.  

\begin{lemma}
\label{lemmaB}
For $d$ large enough,
$$\P(B^d)\geq 1-2dk\exp\left(-2h\left(\frac{1}{\kappa}-
\frac{1}{k}\right)^2\right),$$
where $h=\kappa d^{\alpha}$.
\end{lemma}
\begin{proof}
Let $B_r^d(i,h)$ be the event that $r\in\{1,\ldots,k\}$ 
appears at most $h/\kappa$ times in the string 
$Y_iY_{i+1}\ldots Y_{i+h}$.  Hence, 
$$B^d(i,h)=\bigcap_{r\in\{1,\ldots,k\}}B_r^d(i,h),$$
and since all the symbols have equal probabilities:
\begin{equation}
\label{Bdc}
\P((B^d(i,h))^c)\leq k \P((B_1^d(i,h))^c).
\end{equation}
Now if 
$B_1^{d}(i,h)$ does not hold true, then the letter $1$ 
appears more than $h/\kappa$ times in $Y_iY_{i+1}\ldots Y_{i+h}$.    
Let $W_j$ be the Bernoulli random variable which is equal to one 
if $Y_j=1$ and zero otherwise, and so 
if the event $(B_1^d(i,h))^c$ holds true then so does the event  
\begin{equation}
\label{sum}
\sum_{j=i+1}^{i+h}W_j\geq \frac h{\kappa},\,
\end{equation}
where again by equiprobability, 
$\P(W_j=1)=\E W_j =1/k$.
Hence,
\begin{equation}
\label{PBc2}
1-\P(B_1^{d}(i,h))\le \P\left(
\sum_{j=i+1}^{i+h}W_j-\E\left(\sum_{j=i+1}^{i+h}W_j\right)\geq \frac{h}{\kappa} - \frac{h}{k}\right),
\end{equation}
and since $(1/\kappa)-(1/k)>0$, another use of Hoeffding's inequality leads
to
\begin{equation}
\label{sumo}
1-\P(B_1^{d}(i,h))\leq \exp\left(-2h\left(\frac{1}{\kappa}-
\frac{1}{k}\right)^2\right).
\end{equation}
Since, 
$$(B^d)^c=\bigcup_{i\in[1,2d]}\bigcup_{r\in\{1,\ldots,k\}}(B_r^d(i,h))^c,$$
then
$$\P((B^d)^c)\leq 
\sum_{i\in[1,2d]}\sum_{r\in\{1,\ldots,k\}}
\P((B_r^d(i,h))^c)\le 
2dk\P((B_1^d(1,h))^c),$$
which, with \eqref{sumo}, lead to the announced result:  
$$\P((B^d)^c)\leq 
2dk\exp\left(-2h\left(\frac{1}{\kappa}-
\frac{1}{k}\right)^2\right).$$
\end{proof}

\

As shown now, the event $D^d$ occurs with high probability.

\begin{lemma}
\label{lemmaD}
For $d$ large enough,
$$\P(D^d)\geq 1-4d\exp\left(-d\frac{(\gamma^*_k-\gamma_k^c)^2}{128}\right).$$
\end{lemma}
\begin{proof}
Recall that $D^d=D^d_I\cap D^d_{II}$, where $D^d_I(i)$ is the event that 
there exists an optimal alignment 
of $X$ and $Y$ aligning $X_{d-(\l/2)}$ to $i$.  Now, for $D^d_I$ to hold it is 
enough that none of the events $D^d_I(i)$ hold for all  
$i\notin [i_1,i_2]$.  
Hence, 
$$\bigcap_{i\in[1,2d]\backslash [i_1,i_2]}(D_I^d(i))^c\subset D^d_I,$$ 
so that 
\begin{equation}
\label{Ddc}
\P((D^d_I)^c)\leq \sum_{i\in[1,2d]\backslash [i_1,i_2]}\P(D^{d}_I(i)).
\end{equation}
Let $L(i)$ be the maximal score obtained when leaving out the big block but 
giving as constraint that $X_{d-(\l/2)}$ gets aligned with $i$, i.e.,
\begin{align*}
L(i)&:=|LCS(X_1X_2\cdots X_{d-(\l/2)};Y_1Y_2\cdots Y_i)|\\
&\quad \quad + \quad 
|LCS(X_{d+(\l/2)+1}X_{d+(\l/2)+2}\cdots X_{2d};Y_{i+1}Y_{i+2}
\cdots Y_{2d})|.\end{align*}
As shown next, when $D^d_I(i)$ holds then,  
$$L(i)+2d^\beta\geq 
LC_{2d}^*,$$
where $LC_{2d}^*$ is the length of the LCS 
of $X_1^*X_2^*\ldots X^*_{2d}$ and $Y_1Y_2\ldots Y_{2d}$.  
Indeed, less than $d^\beta$ letters are changed between $X$ and $X^*$, and so the 
length-difference 
between the LCS of $X$ and $Y$ and the LCS of $X^*$ and $Y$ is at most 
$d^\beta$.  Also, if $D^d_I(i)$ holds then 
the difference between $L(i)$ and the length of the LCS of 
$X$ and $Y$ is at most $d^\beta$.  Therefore, the difference between the lengths of 
$L(i)$ and $LC_{2d}^*$ is at most $2d^\beta$, when $D^d_I(i)$ holds.  
Hence, 
\begin{equation}
\label{Ddri}
\P(D^d_I(i))\leq\P(L(i)+2d^\beta\geq 
LC_{2d}^*), 
\end{equation}
with 
\begin{equation}
\label{Ddri2}
 \P(L(i)+2d^\beta\geq 
LC_{2d}^*)=\P(L(i)-LC_{2d}^*-\E L(i) +\E LC_{2d}^* \geq 
\E LC_{2d}^* -\E L(i) -2d^\beta).
\end{equation}
But as $d\rightarrow \infty$, 
$\E LC_{2d}^* /2d \to \gamma^*_k$, and via \eqref{alex} 
\begin{equation}
\label{eld2}\E LC_{2d}^*\geq 2d\gamma_k^*-C_L\sqrt{2d\ln 2d}, 
\end{equation}
for some constant $C_L>0$.  But, by definition, 
\begin{equation}
\label{eli}
\E L(i)=
d_1\gamma_k(p_1,d_1)+d_2\gamma_k(p_2,d_2),
\end{equation}
where 
\begin{align*}
d_1:=&\frac12\left(d-\frac\l2+i\right), \quad p_1:=\frac{i-d+(\l/2)}{d-(\l/2)+i} 
= \frac{2i-2d+\l}{2i+2d-\l},\\
d_2:=&\frac12\left(3d-\frac\l2-i\right), \quad p_2:=\frac{i-d-(\l/2)}{3d-(\l/2)-i}  
= \frac{2i-2d+\l}{6d-2i-\l}. 
\end{align*}
Moreover, 
\begin{equation}
\label{eins}
\gamma_k(p_2,d_2)\leq \gamma^*_k,
\end{equation}
and 
\begin{equation}
\label{zwei}
d_1+d_2=2d-\frac\l2\leq 2d.
\end{equation}
Now, if $i\notin[i_1,i_2]$, then by the very 
definition of $i_1$ and $i_2$, 
\begin{equation}
\label{drei}
\gamma_k(p_1)\leq \gamma^c_k.   
\end{equation} 
Next, by a sub-additivity argument,
$$\gamma_k(p_1)=\lim_{d\rightarrow\infty}\gamma_k(p_1,d)\ge \gamma_k(p_1,d),$$ 
for every $d\ge 1$, and therefore 
\begin{equation}\label{gamma2}
\gamma_k(p_1,d_1)\leq \gamma_k(p_1).  
\end{equation}
Applying \eqref{gamma2}, \eqref{eins}, \eqref{zwei} and \eqref{drei} 
to \eqref{eli}, and assuming that 
$d$ is large enough so that $d_1\geq d/4$, lead to 
\begin{equation}
\label{eli2}  
\E L(i) \leq 
2d\gamma^*_k-\frac{d(\gamma^*_k-\gamma_k^c)}{4}\,.
\end{equation}
Then, \eqref{eli2} and \eqref{eld2} give  
\begin{equation}
\label{l2dli}
\E(LC_{2d}^*-L(i))-2d^\beta\geq \frac{d(\gamma^*_k-\gamma_k^c)}{4}
-C_L\sqrt{2d\ln 2d}-2d^\beta.  
\end{equation}
By definition, $\gamma^*_k-\gamma_k^c>0$ and $\beta<1$.  So, for $d$ large enough, 
the right-hand side of \eqref{l2dli} is at least 
$d(\gamma^*_k-\gamma_k^c)/8$, so that 
\begin{equation}
\label{l3dli}
\E(LC_{2d}^*-L(i))-2d^\beta\geq \frac{d(\gamma^*_k-\gamma_k^c)}{8}\,.
\end{equation}
Using \eqref{l3dli} with \eqref{Ddri} and \eqref{Ddri2}, lead to:   
\begin{equation}
\label{Ddr}
\P(D^d_I(i))\leq
\P\left(L(i)-LC_{2d}^*-\E L(i)+\E L_{2d}^* \geq d\frac{(\gamma^*_k-\gamma_k^c)}{8}
\right), 
\end{equation} 
for $d$ large enough.  By Hoeffding's inequality, 
\begin{equation}
\label{bound}
\P(D^d_I(i))\leq \exp\left(-d\frac{(\gamma^*_k-\gamma_k^c)^2}{128}\right),
\end{equation}
for $i\notin [i_1,i_2]$.   Combining  \eqref{bound} with \eqref{Ddc} gives 
\begin{equation}
\label{Ddc2}
\P((D^d_I)^c)\leq \sum_{i\in[1,2d]\backslash [i_1,i_2]}
\exp\left(-d\frac{(\gamma^*_k-\gamma_k^c)^2}{128}\right)
\leq 2d\exp\left(-d\frac{(\gamma^*_k-\gamma_k^c)^2}{128}\right).
\end{equation}
The same bound can be found for $\P(D^{dc}_{II})$ and this finishes the proof.\end{proof}

\

As the next lemma shows, the event $F^d$ also holds with high probability.

\begin{lemma}
\label{lemmaF}
There exists a constant $C_F > 0$, independent of $d$, such that 
$$\P(F^d)\geq 1-e^{-C_Fd^{2\beta-1}},$$
for all $d\geq 1$.
\end{lemma}
\begin{proof}The proof is very similar 
to the proof that $\P(C^d)$ occurs with high probability 
(see Lemma~\ref{lemmaC} and Lemma~\ref{lemmaC}).  
Nevertheless we present large parts of the proof, since many inequalities are reversed 
and the infinitesimal quantities are of different order. 
Let 
$$\Delta := |LCS(X_1^*X_2^*\ldots X^*_{d-(\l/2)+d^\beta};Y_1Y_2\ldots Y_{i})|
-|LCS(X_1^*X_2^*\ldots X^*_{d-(\l/2)};Y_1Y_2\ldots Y_i)|,$$ 
and assume at first that $i\in [d, i_2]$ (see \eqref{i1}).  
We first show that  
\begin{equation}
\label{hallohallo}
\E\Delta\!\geq \frac{\gamma^b_k}{2}d^\beta.
\end{equation}
As in the proof of Lemma~\ref{lemmaC}, with its notation, 
$$\E\Delta=d_2\gamma_k(d_2,p_2)-d_1\gamma_k(d_1,p_1),$$
where 
\begin{align*}
d_2:=&\frac12 \left(i+d-\frac\l 2+d^\beta\right)\!,  \, 
\quad p_2:=\frac{i-d+(\l/2)-d^\beta}{i+d-(\l/2)+d^\beta},\\
d_1:=&\frac12 \left(i+d-\frac\l 2\right), \quad \quad \quad p_1:=\frac{i-d+(\l/2)}{i+d-(\l/2)}.
\end{align*}
Again, for all $p\in (-1,1)$ and all $d\ge 1$, and since $d_1,d_2\leq 2d$, 
\begin{equation}
\label{Deltad2*}
\E\Delta\geq d_2\gamma_k(p_2)-d_1\gamma_k(p_1)-
2C_\gamma\sqrt{2d \ln 2d}.
\end{equation}
Now, 
\begin{equation}
\label{Deltad3*}
d_2\gamma_k(p_2)-d_1\gamma_k(p_1)
=\frac{d^\beta\gamma_k(p_1)}2+d_2\frac{\delta\gamma_k}{\delta p}\,\delta p,
\end{equation}
where $\delta\gamma_k:=\gamma_k(p_2)-\gamma_k(p_1)$ 
and where $\delta p:=p_2-p_1<0$.  Next, if $p_2\geq 0$, then 
$\delta\gamma_k\geq 0$ and so 
\begin{equation}
\label{I*}
d_2\gamma_k(p_2)-d_1\gamma_k(p_1)\geq \frac{d^\beta\gamma_k(p_1)}2.  
\end{equation}
If $p_2\leq 0$, then 
\begin{equation}
\label{delta_p*}
|\delta p|=
\left|\frac{-2id^\beta}{(i+d-(\l/2))(i+d -(\l/2)+d^\beta)}\right|\leq 2\frac{d^\beta}{d},
\end{equation}
for $d$ large enough, e.g., $i\ge \l/2$, i.e., $(1-q)d \ge \l$.  
Since $i\in[d,i_2]$, $p_1\in [0,q]$.  Moreover, $p_2<p_1$ and $|\delta p|\leq {2d^\beta}/{d}$, 
for $d$ large enough, imply that 
$p_2\in[-q,q]$ and therefore via \eqref{gamma'},  
\begin{equation}
\label{aaaa*}
\frac{|\delta\gamma_k|}{|\delta p|} <  
\frac{(\gamma^c_k-\gamma^b_k)}{16}.
\end{equation}
Combining \eqref{aaaa*} with \eqref{delta_p*} and since $d_2\leq 2d$,
\begin{equation}
\label{additional*}
\frac{d^\beta\gamma_k(p_1)}2+d_2\,\frac{|\delta\gamma_k|}{\delta p}\,\delta p
\ge d^\beta\left(\frac{\gamma_k(p_1)}2-\frac{(\gamma^c_k-\gamma^b_k)}{4}\right).
\end{equation}
Now, $i\in[i_1,i_2]$ hence $p_1\in[-q,q]$ and therefore, by the very definition of $i_1,i_2$,
$\gamma_k(p_1)\geq \gamma_k^c$, which yields
\begin{equation}\label{eq.xx*}
d^\beta\left(\frac{\gamma_k(p_1)}2-\frac{(\gamma^c_k-\gamma^b_k)}{4}\right)
\ge
d^\beta
\left(\frac{\gamma_k^b}2+\frac{(\gamma^c_k-\gamma^b_k)}{4}\right).\end{equation}
Now, \eqref{eq.xx*} together with \eqref{additional*},
\eqref{Deltad2*} and \eqref{Deltad3*} imply that
\begin{equation}
\label{Deltad4*}
\E\Delta\ge d^{\beta}\frac{\gamma^c_k}{2}+d^\beta\frac{(\gamma_k^c-\gamma_k^b)}{4}-
2C_\gamma\sqrt{2d \ln 2d}.
\end{equation}
Finally, since $\beta>1/2$ is independent of
$d$, and since
$\gamma_k^c-\gamma_k^b>0$, then $2C_\gamma\sqrt{2d \ln 2d}$ becomes ``negligible" when compared 
to $d^{\beta}(\gamma_k^c-\gamma_k^b)/4$.  So, for large enough 
$d$, \eqref{Deltad4*} implies that 
$\E\Delta\geq  d^{\beta}{\gamma^b_k}/{2}$,
which is what we intended to prove for $i\in [d, i_2]$.  

Next, for $i\in [i_1, d]$,
and $d$ large enough, the inequality \eqref{hallohallo} remains valid.  Indeed,
first, one only needs $i\in [d,i_2]$ instead of $i\in[i_1,i_2]$, to obtain 
\eqref{aaaa*}; more specifically
one needs $p_2\in[-q,q]$. 
However, \eqref{gamma'} is a strict inequality and so, for 
$d$ large enough, even if $p_2 \notin [-q,q]$ but as long as $i\in[i_i,i_2]$, 
by continuity, recalling also that $|\delta p| \le 2d^\beta/d$, 
and since $\gamma^\prime_k(q^+)=\gamma^\prime_k(q^-)$, the inequality 
\eqref{aaaa*} continues to hold.  
This will then imply \eqref{hallohallo} for all $i\in[i_1,i_2]$.  

We wish now to upper-bound the probability of the complement of $F^{d}_i$ 
when $i\in[i_1,i_2]$.  From \eqref{hallohallo}, 
\begin{equation}
\label{Fnc101}
\P((F^{d}_i)^c)=\P\left(\Delta \leq d^\beta\frac{\gamma^a_k}{2}\right)
\leq 
\P\left(\Delta-\E\Delta\leq d^\beta\frac{\gamma^a_k-\gamma^b_k}{2}\right).
\end{equation} 
Note that $\Delta$ depends on $d^*:=d-(l/2)+d^\beta+i$, iid entries 
$X_1^*,X_2^*,\ldots, X^*_{d-(\l/2)+d^\beta}$ and $Y_1,Y_2,\ldots, Y_{i}$, and so 
by Hoeffding's inequality, 
\begin{align*}
\P\left(\Delta-\E\Delta\leq d^\beta\frac{\gamma^a_k-\gamma^b_k}{2}\right)
&\leq \exp\left(-\frac{(d^\beta(\gamma^a_k-\gamma^b_k))^2}{8d^*}\right)\\
&\leq  \exp\left(-d^{2\beta-1}\frac{(\gamma^b_k-\gamma^a_k)^2}{32}\right),
\end{align*} 
since for $d$ large enough, $d^*\leq 4d$.
Hence, since $[i_1,i_2]$ contains less than $d$ elements 
(for $d$ large enough), 
$$\P((F^{d})^c)\leq\sum_{i\in[i_1,i_2]}\P((F^{d}_i)^c)\le 
d\exp\left(-d^{2\beta-1}\frac{(\gamma^b_k-\gamma^a_k)^2}{32}\right).$$
Finally, recall that $\beta > 1/2$.   
\end{proof}

\paragraph{Proof of Theorem~\ref{maintheoremsubsection}.}
This is the main theorem for three or more letters, i.e., for $\gamma^*_k/2>1/k$.  
It states that the events $E^d$ and $K^d$ 
both hold high probability.  Hence, for $d$ large enough, 
typically with three or more letters the long block gets mainly aligned with gaps.  
Moreover, this result asserts that replacing the long block with iid symbols typically 
leads to an increase in the LCS which is linear in the length of the long block. 

Let us first handle $E^d$.  By Lemma~\ref{combi1}, 
\begin{equation}
\label{ABC}
\P((E^d)^c)\leq \P((B^d)^c)+\P((C^d)^c)+\P((D^d)^c).
\end{equation}
By Lemma~\ref{lemmaC}, $\P((C^d)^c)$ is of exponential small 
order in $d^{2\alpha-1}$; by Lemma~\ref{lemmaB}, 
$\P((B^d)^c)$ is of exponential small order in $d^{\alpha}$ and by 
Lemma~\ref{lemmaD} $\P((D^d)^c)$ is exponentially 
small in $d$.  Therefore, for $\alpha\in\,(1/2,1)$, 
$\P((E^d)^c)$ is also exponentially small in 
$d^{2\alpha-1}$.  Hence, there exists a constant $C_E>0$, 
independent of $d$ (but depending on $k$) such that 
$$\P((E^d)^c)\leq e^{-C_Ed^{2\alpha-1}}.$$  
Let us, next, turn our attention to the event $K^d$.  
From Lemma~\ref{combiO}, 
\begin{equation}
\label{ABC2}
\P((K^d)^c)\leq \P((B^d)^c)+\P((C^d)^c)+\P((D^d)^c)+\P((F^d)^c), 
\end{equation}
and, as already seen,  
$\P((B^d)^c)+\P((C^d)^c)+\P((D^d)^c)$ is of exponential 
small order in $d^{2\alpha-1}$.  By Lemma~\ref{lemmaF}, 
$\P((F^d)^c)$ is of exponential small order in 
$d^{2\beta-1}$.  Since $2\alpha-1<2\beta-1$, the right side 
of \eqref{ABC2} is thus exponentially small 
in $d^{2\alpha-1}$, and therefore there exists a constant 
$C_K>0$, independent of $d$, such that 
$$\P((K^d)^c)\leq e^{-C_Kd^{2\alpha-1}}.$$
\hfill~\rule{0.5em}{0.5em}

\subsection{Proofs For Binary Strings}\label{rigorous2}
The purpose of this subsection is to prove 
Theorem~\ref{maintheoremsubsection2}, and therefore throughout the rest of the article, 
$k\gamma^*_k<2$, i.e., $k=2$.  Theorem~\ref{maintheoremsubsection2} states 
that typically, for $d$ large enough, the long block gets mainly aligned
with symbols and not with gaps.  The corresponding event $G^d$  was defined in 
Section~\ref{statements}.  Theorem~\ref{maintheoremsubsection2} also asserts that 
replacing the long block with iid symbols typically increases the LCS linearly in the length 
of the long constant block.  The corresponding event $H^d$ was also defined in 
Section~\ref{statements}.  So, below, we intend to prove 
that both events hold with high probability and this is done in a way very similar 
to the $3$-or more letter-case.  

Let $k_{II}$ and 
$\gamma_2^{II}$ be two constants, independent of $d$, such that 
$k_{II}>2$ and $\gamma_2^{II}>\gamma^*_2$, but also such that 
$k_{II}\gamma_2^{II}<2$ (this last choice is certainly possible since 
$\gamma^*_2<1$).  Actually, for the argument which follows, any values $k_{II}>2$ 
and $\gamma_2^{II}>\gamma^*_2$ will do, provided 
the constants are close enough to their respective bounds 
and do not depend on $d$.

Let now $B^d_{II}$ be the event that in any piece of 
$Y$ of length $k_{II}d^{\alpha}$, there are at least $d^{\alpha}$ 
ones and zeros.  More precisely, let  
$B^d_{II}(i)$ be the event that 
$$ \sum_{j=i+1}^{i+h}Y_j\geq d^{\alpha},$$ 
and that 
$$\sum_{j=i+1}^{i+h}|Y_j-1|\geq d^{\alpha},$$
where $h:=k_{II}d^{\alpha}$.  Finally, let 
\begin{equation*} 
B^d_{II}=\bigcap_{i=1}^{2d-h}B^d_{II}(i).
\end{equation*}

Let $C^d_{II}$ be the event that an increase of the length of 
$Y_1\ldots Y_i $ by $k_{II}d^{\alpha}$ 
leads to an increase of the LCS of $Y_1\ldots Y_i$ 
and $X_1\ldots X_{d-(\l/2)}$ of no more than $k_{II}d^{\alpha}\gamma_2^{II}/2$ 
for all $i+k_{II}d^{\alpha}\in [i_1,i_2]$.  More precisely, for $h=k_{II}d^{\alpha}$, 
$$C^d_{II}(i)\!:=\left\{|LCS(X_1\ldots X_{d-(\l/2)};Y_1\ldots Y_{i+h})| 
- |LCS(X_1\ldots X_{d-(\l/2)};Y_1\ldots Y_i)|\leq \frac{h\gamma_2^{II}}2\right\},$$ 
$$C^d_{II}:=\bigcap_{i+h\in[i_1,i_2]} C^d_{II}(i).$$ 

Recall finally that $G^d$ is the event that the long constant block gets mainly  
aligned with symbols and not with gaps; more precisely, 
$$G^d=\left\{|LCS(X;Y)|>|LCS(X_1X_2\ldots X_{d-(\l/2)}X_{d-(\l/2)+d^{\alpha}+1}
X_{d-(\l/2)+d^{\alpha}+2}\ldots X_{2d};Y)|\right\}.$$

\begin{lemma}
\label{subsetG} 
$$B_{II}^d\cap D^d\cap C^d_{II}\subset G^d.$$
\end{lemma}
\begin{proof}The proof is by contradiction.  Assume that $\pi$ is an optimal alignment 
of $X$ and $Y$ aligning at least $d^{\alpha}$ symbols from the long block with gaps,  
and assume that $\pi$ aligns $X_{d-(\l/2)}$ with $j$.  
Since $D^d$ holds, then $j\in[i_1,i_2]$.  Now, let $i:=j-h$ (again, $h:=k_{II}d^{\alpha}$), so that 
$i+h\in[i_1,i_2]$.  Thus $C^d_{II}$ ``applies'' 
to $i$, meaning that when ``taking out'' the piece 
$Y_{i+1}Y_{i+2}\ldots Y_j$ from the alignment $\pi$, lose at most 
$h\gamma_2^{II}/2$.  Now, because of 
$B^d_{II}$, the string $Y_{i+1}Y_{i+2}\ldots Y_{i+h}$ contains 
the symbols the long block is made of, at least $d^{\alpha}$  
times.  Hence the $d^{\alpha}$ symbols, from 
the long block, which are aligned by $\pi$ with gaps, 
can be aligned with symbols contained in the piece of string 
$Y_iY_{i+1}\ldots Y_{i+h}$.  Let $\bar\pi$ denote the new alignment  
obtained from modifying $\pi$ in this way.  Transforming 
$\pi$ to $\bar\pi$ gained $d^{\alpha}$ 
aligned symbols from the long block, which where aligned with gaps,  
and now are aligned with symbols.  However, from the previously aligned 
symbols from $Y_i\ldots Y_{i+h}$ we could lose as many 
as $h\gamma_2^{II}/2$ aligned symbols pairs.  
Hence the change is at least 
\begin{equation}
\label{increaseb}
d^{\alpha}-\frac{d^{\alpha}k_{II}\gamma_2^{II}}2
=d^{\alpha}\left( 1-\frac{k_{II}\gamma_2^{II}}{2}\right) > 0,  
\end{equation}
from the choices of $k_{II}$ and $\gamma_2^{II}$.  Hence, \eqref{increaseb} is strictly 
positive, and thus $\bar\pi$ aligns more letter-pairs than $\pi$.  
Therefore, $\pi$ is not optimal, which is a contradiction, and it is not possible 
for $d^\alpha$ symbols, of the long constant block, to get aligned with gaps when  
$B^d_{II}$, $D^d$ and $C^d_{II}$ all hold.  
Hence, $B^d_{II}$, $D^d$ and $C^d_{II}$ jointly imply $G^d$.
\end{proof}

\paragraph{High probability of $G^d$.}  From Lemma~\ref{subsetG}, 
\begin{equation}
\label{PG13}
\P((G^d)^c)\leq \P((B_{II}^d)^c)+\P((D^d)^c)+\P((C^d_{II})^c).
\end{equation}
In  Lemma~\ref{lemmaD}, we already proved that $\P((D^d)^c)$ is exponentially small 
in $d$.  Next, a simple application of Hoeffding's inequality 
shows that $\P((B^{d}_{II})^{c})$ is exponentially small in 
$d^{\alpha}$ and this is left to the reader.   
Let us now deal with $\P(C^d_{II})$ and show  
that $\P((C^d_{II})^c)$ is exponentially small in 
$d^{2\alpha-1}$.  The proof is similar to the proof 
of Lemma~\ref{difficult}.  Using the notations there, but with $h=k_{II}d^{\alpha}$, let   
$$\Delta:=|LCS(X_1^*X_2^*\ldots X^*_{d-(\l/2)};Y_1Y_2\ldots Y_{i+h})|-
|LCS(X_1^*X_2^*\ldots X^*_{d-(\l/2)};Y_1Y_2\ldots Y_{i})|.$$
Again, 
$$\E\Delta=d_2\gamma_2(d_2,p_2)-d_1\gamma_2(d_1,p_1)$$
where $d_1,d_2,p_1,p_2$ are as in Lemma~\ref{difficult}, and 
\begin{equation}
\label{newnew}
\left|\E\Delta-d_2\gamma_2(p_2)+d_1\gamma_2(p_1)\right|\leq
2C_\gamma\sqrt{2d\ln 2d}.  
\end{equation}
Once more, 
\begin{equation}
\label{banana}
d_2\gamma_2(p_2)-d_1\gamma_2(p_1)=
\frac{h\gamma_2(p_1)}{2}+d_2\frac{\delta\gamma_2}{\delta p}\delta p,
\end{equation}
where 
$$\delta\gamma_2:=\gamma_2(p_2)-\gamma_2(p_1),$$
and where, for $d$ large enough, 
\begin{equation}
\label{dunkindonut}0< \delta p:= p_2 - p_1 \leq 2\frac{h}{d}.  
\end{equation}
Since $i\in[i_1,i_2]$ then $p_1\in[-q,q]$.  
If $p_2$ would also be in $[-q,q]$, then 
the inequality \eqref{gamma'bis} would be enough to get our estimates.  
Now, $p_2$ might not be 
in $[-q,q]$, but by continuity in $d$ and since 
$\delta p\rightarrow 0$, as $d\rightarrow\infty$, we have 
for large enough $d$:  
\begin{equation}
\label{gatech}
\frac{|\delta\gamma_2|}{\delta p}\leq 
\frac{\gamma^{II}_2 - \gamma_2^*}{16} 
\end{equation}
Combining  \eqref{newnew}, \eqref{banana}, \eqref{dunkindonut}, 
\eqref{gatech} and \eqref{gamma'bis} 
with the facts that $d_2\leq 2d$ and $\gamma_2(p_1)\leq \gamma_2^*$ lead to:
\begin{equation}
\label{tutu}
\E\Delta-\frac{h\gamma_2^*}{2} \le 
 \frac{h}{4}(\gamma_2^{II}-\gamma_2^*) + 2C_\gamma\sqrt{2d\ln 2d}.   
\end{equation}
Applying Hoeffding's inequality to $\Delta$,  
which depends on $d-(\l/2)+i+h<4d$ iid entries and using \eqref{tutu} yield: 
\begin{align*}
\P((C^{d}_{II}(i))^c) = \P\left(\!\Delta - \E\Delta > \frac{h\gamma_2^{II}}{2} - \E\Delta\!\right) 
&\le \P\left(\!\Delta - \E\Delta > \frac{h(\gamma_2^{II} - \gamma_2^*)}{4} - 
2C_\gamma\!\sqrt{2d\ln 2d}\right)\\
&\le \exp\left(-\frac{k_{II}^2(\gamma_2^{II}-\gamma_2^*)^2}{512}d^{2\alpha-1}\right), 
\end{align*}
since for $d$ large enough, the term $2C_\gamma\sqrt{2d\ln 2d}$ 
becomes negligible when compared to $h=k_{II}d^{\alpha}$, $\alpha > 1/2$.  Next, $[i_1,i_2]$ contains 
at most $2d$ integers and so 
$\P((C^{d}_{II})^c)\leq 2d
\exp\left(-k_{II}^2(\gamma_2^{II}-\gamma_2^*)^2d^{2\alpha-1}/512\right)$.  
Finally, $0<2\alpha-1<\alpha<\beta<1$ and, therefore, the orders of magnitude 
of $\P((B^{d}_{II})^{c})$, $\P((C^{d}_{II})^{c})$ and $\P((D^d)^c)$ 
together with \eqref{PG13} imply that 
$$\P((G^d)^c)\leq e^{-C_G d^{2\alpha-1}},$$
for all $d\ge 1$, where $C_G>0$ is a constant independent 
of $d$.  This finishes establishing that, with high probability, 
a small proportion of gaps is aligned with the long block.
\hfill~\rule{0.5em}{0.5em}

\

\

The rest of this subsection is devoted to analyzing the 
increase in the LCS when replacing the long constant 
block with iid symbols, thus showing that the event $H^d$ 
holds with high probability.  Recall that 
$H^d$ states that the increase in the 
LCS is at least ${\tilde c}_H>0$ times the length of the long block.
(${\tilde c}_H$ is any positive real, independent of $d$, smaller than 
$3\gamma^*_2/2-1$.)  Again, $X$ contains a long block in 
$[d-\l/2 + 1,d+\l/2]$, i.e., 
$$\P\left(X_i=X_{i+1},\quad\forall i\in\left[d-\left(\frac\l 2\right)+1,
d+\left(\frac\l 2\right)-1\right]\right)=1,$$
while $X^*$ is obtained by replacing the long block 
in $X$ by iid symbols, i.e., $X^*_i=X_i$ for all 
$i\notin[d-(\l/2)+1,d+(\l/2)]$.  Let now $\tilde k_{II} < 2$ and 
$\tilde\gamma_2 < \gamma^*_2$ be two constants independent 
of $d$ such that $\tilde k_{II}$ is extremely close to $2$ while 
$\tilde \gamma_2$ extremely close to $\gamma^*_2$, with moreover 
\begin{equation}
\label{kII*}\left(\frac{(1+\tilde k_{II})\tilde\gamma_2}{2}-1\right)>{\tilde c}_H.  
\end{equation}
(These choices are certainly possible 
since $3\gamma^*_2/2-1 \approx 0.2$ 
and since (see \eqref{ch}) ${\tilde c}_H < 3\gamma^*_2/2-1$.)

Next, let $\tilde B^{d}_{II}$ be the event that in any piece of 
$Y$ of length $\tilde k_{II}d^{\beta}$ there are strictly less than 
$d^{\beta}-d^{\alpha}$ zeros and ones.
More precisely, for $h:=\tilde k_{II}d^\beta$, 
$$\tilde B^{d}_{II}(i):= \left\{\sum_{j=i+1}^{i+h}Y_j< d^{\beta}-d^{\alpha}, \quad 
\sum_{j=i+1}^{i+h}|Y_j-1|< d^{\beta}-d^{\alpha}\right\}.$$
Let also  
$$\tilde B^{d}_{II}=\bigcap_{i=1}^{2d-h}\tilde B^{d}_{II}(i).$$

Let $\tilde C^{d}_{II}$ be the event that 
an increase of the length of $Y_1\ldots Y_i$ by $\tilde k_{II}d^{\beta}$ 
and an increase of the length of $X_1\ldots X_{d-(\l/2)}$ by $d^\beta$ leads to an increase 
of the LCS by at least $d^\beta(1+\tilde k_{II})\tilde \gamma_2/2$, 
for all $i\in [i_1,i_2]$.  More precisely,  let $\tilde C^{d}_{II}(i)$ be the event that 
$$|LCS(X_1\ldots X_{d-(\l/2)+d^\beta};Y_1\ldots Y_{i+h_y})|
-|LCS(X_1\ldots X_{d-(\l/2)};Y_1\ldots Y_i)|\geq
d^\beta\!\left(\!\!\frac{(1+\tilde k_{II})\tilde \gamma_2}{2}\!\!\right)\!\!,$$
where, again, $h:=\tilde k_{II}d^\beta$, and let 
$$\tilde C^{d}_{II}:=\bigcap_{i\in [i_1,i_2]} \tilde C^{d}_{II}(i).$$   
Once more, $H^d$ is the event that replacing the long constant block 
by iid symbols increases the LCS by at least ${\tilde c}_H d^\beta$; more precisely, 
$$H^d=\left\{|LCS(X^*;Y)|
-|LCS(X;Y)|\geq {\tilde c}_H d^\beta\right\}.
$$

\

\begin{lemma}
\label{Bd*}
$$\tilde B^{d}_{II}\cap \tilde C^{d}_{II}\cap D^d\subset H^d.$$
\end{lemma}
\begin{proof} Assume that $\tilde B^{d}_{II}$, $\tilde C^{d}_{II}$ and $D^d$ all hold true 
and thus $G^d$ also holds.  Let now $\pi$ be an optimal alignment.  Hence, 
$\pi$ aligns at least $d^\beta-d^{\alpha}$ symbols from the long constant 
block with symbols.  Let $[i,j]$ denote the interval 
on which the long block gets aligned to by $\pi$, meaning 
that $X_{d-(\l/2)+1}$ gets aligned to $i$ by $\pi$ while $X_{d+(\l/2)}$
gets aligned to $j$.   Next, at least 
$d^\beta-d^{\alpha}$ symbols are aligned with symbols from the long block, it follows 
via the event $\tilde B^{d}_{II}$, that $j-1 \ge \tilde k_{II}d^{\beta}$ 
(in order for $[i,j]$ to contain sufficiently many same symbols).  
Now modify the alignment $\pi$ to obtain an alignment $\bar\pi$ 
aligning $X^*$ and $Y$.  The new alignment $\bar\pi$ 
is identical to $\pi$ in the way it aligns 
$X_{d+(\l/2)+1}X_{d+(\l/2)+2}\ldots X_{2d}$ with $Y_{j+1}Y_{j+2}\ldots Y_{2d}$, 
but instead of aligning the long block to $Y_iY_{i+1}\ldots Y_j$, 
it now aligns $X^*_1X_2^*\ldots X^*_{d+(\l/2)}$ with $Y_1Y_2\ldots Y_{j}$. 
Since $D^d$ holds, then $i\in[i_1,i_2]$ and therefore we can apply 
$\tilde C^{d}_{II}$ to $i$.  This yields that the 
gain by aligning  $X^*_1X_2^*\ldots X^*_{d+(\l/2)}$ with $Y_1Y_2\ldots Y_{j}$, 
instead of just aligning $X^*_1X_2^*\ldots X^*_{d-(\l/2)}$ with 
$Y_1Y_2\ldots Y_{i}$, is equal to 
$d^\beta(1+\tilde k_{II})\tilde\gamma_2/2$.  
On the other hand, there is a loss of at most $d^\beta$ symbols from the long 
constant block, so the overall gain is of at least:  
 $$d^\beta\left(\frac{(1+\tilde k_{II})\tilde \gamma_2}{2}-1\right).$$
Therefore the event $H^d$ holds true and this finishes this proof.  
\end{proof}

\paragraph{High probability of $H^d$.}
From Lemma~\ref{Bd*}, 
\begin{equation}
\label{hdc}
\P((H^d)^c)\leq \P((\tilde B^d_{II})^{c})+ \P((\tilde C^{d}_{II})^{c})+\P((D^d)^c),
\end{equation}
and clearly, with the help of Lemma~\ref{lemmaD}, we only need to estimate 
$\P(\tilde C^d_{II})$ and $\P(\tilde B^d_{II})$.  A simple application of 
Hoeffding's inequality, left to the reader, shows that  
$\P((\tilde B^d_{II})^c)$ is exponentially small in 
$d^\beta$.  For $\P(\tilde C^d_{II})$, let 
$$\Delta:=|LCS(X_1^*X_2^*\ldots X^*_{d-(\l/2)+d^\beta};Y_1Y_2\ldots Y_{i+h})|-
|LCS(X_1^*X_2^*\ldots X^*_{d-(\l/2)};Y_1Y_2\ldots Y_{i})|,$$
where $h=\tilde{k}_{II}d^\beta$.  
Again,  $\E\Delta=d_2\gamma_2(d_2,p_2)-d_1\gamma_2(d_1,p_1)$, 
where $d_1,p_1$ are as in the proof of Lemma~\ref{difficult} but where $d_2,p_2$ are different, i.e.,  
\begin{align*}
d_2:=&\frac12 \left(i+h+d-\frac{\l}{2} +d^\beta\right)\!,  \, 
\quad p_2:=\frac{i+h-d+ (\l/2)-d^\beta}{i+h+d-(\l/2)+d^\beta},\\
d_1:=&\frac12 \left(i+d-\frac\l 2\right), \quad \quad \quad \quad \quad p_1:=\frac{i-d+(\l/2)}{i+d-(\l/2)}. 
\end{align*}
Once more, 
\begin{equation}
\label{newnew2}
\left|\E\Delta-d_2\gamma_2(p_2)+d_1\gamma_2(p_1)\right|\leq
2C_\gamma\sqrt{2d\ln 2d},  
\end{equation}
\begin{equation}
\label{banana2}
d_2\gamma_2(p_2)-d_1\gamma_2(p_1)=
\frac{(d^\beta+h)\gamma_2(p_1)}{2}+d_2\frac{\delta\gamma_2}{\delta p}\delta p,
\end{equation}
where, for $d$ large enough, 
\begin{equation}
\label{dunkindonut2} 0< |\delta p:= p_2 - p_1| \leq 2\frac{h+d^\beta}{d}= 
2(1+\tilde{k}_{II})\frac{d^\beta}{d}.  
\end{equation}
Since $i\in[i_1,i_2]$ then $p_1\in[-q,q]$.  
If $p_2$ would also be in $[-q,q]$, then 
the inequality \eqref{gamma'bisbis} would be enough to get our estimates.  
Now, $p_2$ might not be in $[-q,q]$, but by continuity in $d$ and since 
$\delta p\rightarrow 0$, as $d\rightarrow\infty$, we have 
for large enough $d$:  
\begin{equation}
\label{gatech2}
\frac{|\delta\gamma_2|}{\delta p}\le 
\frac{(\gamma_2^c-\tilde{\gamma_2})}{32}
\end{equation}
Combining  \eqref{newnew2}, \eqref{banana2}, \eqref{dunkindonut2}, \eqref{gatech2} and \eqref{gamma'bisbis}  
with the facts that $d_2\leq 2d$ and $\gamma_2(p_1)\ge \gamma_2^c$ lead to:
\begin{equation}
\label{tutu3}
\E\Delta-\frac{(h+d^\beta)\gamma_2^c}{2}\ge
\E\Delta-\frac{(h+d^\beta)\gamma_2(p_1)}{2}\ge
 -\frac{h+d^\beta}{8}(\gamma_2^c-\tilde{\gamma}_2)-
2C_\gamma\sqrt{2d\ln 2d}.
\end{equation}
When $d$ is large enough the term $2C_\gamma\sqrt{2d\ln 2d}$ 
becomes negligible when compared to $h=O(d^{\beta})$, $\beta>1/2$.
Hence, for $d$ large enough,
we find 
\begin{equation}
\label{tutu4}
\E\Delta-\frac{(h+d^\beta)\tilde{\gamma_2}}{2}\geq
 \frac{h+d^\beta}{4}(\gamma_2^c-\tilde{\gamma}_2).
\end{equation}
Applying Hoeffding's inequality to $\Delta$, which depends 
on $d-(\l/2)+d^\beta+i+h<5d$ iid entries, yields:  
$$\P((\tilde{C}^{d}_{II}(i))^c)
\leq 
\exp(-d^{2\beta-1}(\tilde{k}_{II}+1)^2(\gamma_2^c-\tilde{\gamma}_2)^2/1024).
$$
Next, $[i_1,i_2]$ contains at most $2d$ integers, 
and so 
$$\P((\tilde{C}^{d}_{II})^c)
\le 2d\exp(-d^{2\beta-1}(\tilde{k}_{II}+1)^2(\gamma_2^c-\tilde{\gamma}_2)^2/1024).
$$
Finally, $0<2\alpha-1<\alpha<\beta<1$ and, therefore, the orders of magnitude 
of $\P((\tilde B^{d}_{II})^{c})$, $\P((\tilde C^{d}_{II})^{c})$ and $\P((D^d)^c)$ 
together with \eqref{hdc} imply that 
$$\P((H^d)^c)\leq e^{-C_H d^{2\alpha-1}},$$
for all $d\ge 1$, where $C_H>0$ is a constant independent 
of $d$.  This finishes establishing that, with high probability, 
replacing the long constant block by iid symbols increases the LCS.
\hfill~\rule{0.5em}{0.5em}

\

\noindent {\bf Acknowledgments.} It is a pleasure to thank both referees for their numerous 
detailed and thoughtful comments on the 
manuscript leading to the current version.  In particular, one of the referees 
suggested we replaced our differentiability condition at every maxima of the mean LCS function by 
a nicer non-tangential (cone) differentiability condition which can be  
verified up to a given degree of confidence using Monte Carlo simulations.

\

\end{document}